%% file: Plancherel-hurwitz.tex
\documentclass[aop,preprint]{imsart}

\RequirePackage{amsthm,amsmath,amsfonts,amssymb,bm}
\RequirePackage[numbers]{natbib}
\RequirePackage[colorlinks,citecolor=blue,urlcolor=blue]{hyperref}
\RequirePackage{graphicx}

\startlocaldefs

\theoremstyle{plain}
\newtheorem{axiom}{Axiom}
\newtheorem{claim}[axiom]{Claim}
\newtheorem{theorem}{Theorem}[section]
\newtheorem{lemma}[theorem]{Lemma}
\newtheorem{corr}[theorem]{Corollary}
\newtheorem{prop}[theorem]{Proposition}
\theoremstyle{remark}
\newtheorem{definition}[theorem]{Definition}

\usepackage{esint}

\newlength{\sizeofpage} 
\usepackage[utf8]{inputenc}

\def\DMO{\DeclareMathOperator}
\DMO{\GL}{GL}
\DMO{\SL}{SL}
\DMO{\SO}{SO}
\DMO{\SU}{SU}
\DMO{\ad}{ad}
\DMO{\Ad}{Ad}
\DMO{\id}{id}
\DMO{\der}{Der}
\DMO{\gl}{\mathfrak{gl}}

\DMO{\su}{\mathfrak{su}}

\renewcommand{\P}{\mathbb{P}}
\newcommand{\Z}{Z}
\newcommand{\Zl }{\Z(\{\lambda\})}
\newcommand{\eps}{\varepsilon}
\newcommand{\lp}{{\lambda^+}}
\newcommand{\lm}{{\lambda^-}}
\newcommand{\lone}{\lambda_1}
\newcommand{\Rl}{R_\lambda}
\newcommand{\Cl}{C_\lambda}
\newcommand{\hl}{h_\lambda}
\newcommand{\fl}{f_\lambda}
\newcommand{\lnn}{\lambda\vdash n}
\def\lh{\bm\lambda}
\def\intvl{\left[\frac{\mt \ell}{\log n}, \frac{\Mt \ell}{\log n}\right]}

\newcommand{\Mt}{6}
\newcommand{\mt}{0.4}

\def\ll{\ell(\lambda)}
\def\tl{{\tilde \lambda}}

\def\H{\mathcal H}

\newcommand\contentsetm[1]{\ensuremath{\mathcal X^{#1}_m}}
\newcommand\firstsetm[1]{\ensuremath{\mathcal Y^{#1}_m}}

\endlocaldefs

\begin{document}

\begin{frontmatter}
\title{Random partitions under the Plancherel--Hurwitz measure, \\ high genus Hurwitz numbers and maps}

\runtitle{Random partitions under the Plancherel--Hurwitz measure}

\begin{aug}

\author[A]{\fnms{Guillaume}~\snm{Chapuy}\ead[label=e1]{guillaume.chapuy@irif.fr}},
\author[B]{\fnms{Baptiste}~\snm{Louf}\ead[label=e2]{baptiste.louf@math.uu.se}}
\and
\author[A,C]{\fnms{Harriet}~\snm{Walsh}\ead[label=e3]{harriet.walsh@ens-lyon.fr}}

\runauthor{G. Chapuy, B. Louf and H. Walsh}

\address[A]{Université Paris Cité, IRIF, CNRS, F-75013 Paris, France \printead[presep={,\ }]{e1}}

\address[B]{Department of Mathematics, Uppsala University, PO Box 480, SE-751 06 Uppsala, Sweden \printead[presep={,\ }]{e2}}

\address[C]{Univ Lyon, ENS de Lyon, CNRS, Laboratoire de Physique, F-69342 Lyon, France\printead[presep={,\ }]{e3}}
\end{aug}

\begin{abstract}

We study the asymptotic behaviour of random integer partitions under a new probability law that we introduce, the Plancherel--Hurwitz measure. This distribution, which has a natural definition in terms of Young tableaux, is a  deformation of the classical Plancherel measure which appears naturally in the context of Hurwitz numbers, enumerating certain transposition factorisations in symmetric groups.

We study a regime in which the number of factors in the underlying factorisations grows linearly with the order of the group, and the corresponding topological objects, Hurwitz maps, are of high genus. We prove that the limiting behaviour exhibits a new, twofold, phenomenon: the first part becomes very large, while the rest of the partition has the standard Vershik--Kerov--Logan--Shepp limit shape. As a consequence, we obtain asymptotic estimates for unconnected Hurwitz numbers with linear Euler characteristic, which we use to study random Hurwitz maps in this regime. This result can also be interpreted as the return probability of the transposition random walk on the symmetric group after linearly many steps.
\end{abstract}

\begin{keyword}
\kwd{Random partitions}
\kwd{limit shapes}
\kwd{transposition factorisations}
\kwd{map enumeration}
\kwd{Hurwitz numbers}
\kwd{RSK algorithm}
\kwd{giant components} \\
This project has received funding from the European Research Council (ERC) under
the European Union’s Horizon 2020 research and innovation programme (grant agreement No.
ERC-2016-STG 716083 “CombiTop”) 
\end{keyword}

\end{frontmatter}

\section{Introduction and main results}

\subsection{Random partitions, Plancherel and Plancherel--Hurwitz measures}

Let $n\geq 1$ be an integer and let $\mathfrak{S}_n$ denote the group of permutations on $[n]:=\{1,2,\dots,n\}$. The famous \emph{Robinson--Schensted algorithm}~\cite{Robinson_1938,Schensted_1961} (see e.g.~\cite{Stanley:EC2}; we abbreviate it as \emph{RSK}, referring to the generalisation by Knuth~\cite{Knuth_1970}) associates each permutation $\sigma\in \mathfrak{S}_n$ bijectively to a pair $(P,Q)$ of standard Young tableaux (SYT) of the same shape. It is impossible to overstate the importance of this construction in enumerative and algebraic combinatorics. At the enumerative level, the RSK algorithm gives a bijective proof of the following  identity:
\begin{align}\label{eq:RSK}
	\sum_{\lambda \vdash n} {f_\lambda}^2 =n!,
\end{align}
where the sum is taken over integer partitions of $n$, and where $f_\lambda$ is the number of SYT of shape $\lambda$ (see Section~\ref{sec:definitions} for full definitions).
If the permutation $\sigma$ is chosen uniformly at random, the shape $\lambda$ of the associated tableaux is a random partition of $n$ distributed according to the probability measure
\begin{equation}
\lambda \mapsto \frac{1}{n!} {f_\lambda}^2,
\end{equation}
which is the \emph{Plancherel measure} of the symmetric group $\mathfrak{S}_n$.

The study of random partitions under the Plancherel measure is an immense subject in itself with many ramifications. One of the classical and most famous results is the fact, due independently to Logan and Shepp~\cite{Logan_Shepp_1977} and Vershik and Kerov~\cite{Vershik_Kerov_1977}, that when $n$ goes to infinity, the diagram of a Plancherel distributed partition $\lh$ converges in some precise sense to a deterministic limit shape (Theorem~\ref{thm:vkls} below) that we call the \emph{VKLS limit shape} following these authors' initials. Other profound results deal with the behaviour of the largest part $\lh_1$, which coincides with the \emph{longest increasing subsequence} of the random permutation $\sigma$, which scales as $2\sqrt{n}$ with fluctuations of order $n^{1/6}$ driven by a Tracy--Widom distribution~\cite{Baik_Deift_Johansson_1999}. 
The book by Romik~\cite{Romik_2015} is a delightful introduction to the subject.

\begin{figure}
    \begin{center}
{\def\svgwidth{0.8\sizeofpage}
\input{for-highgenus-sim} 
}\end{center}
    \caption{A random partition $\boldsymbol{\lambda}$ of $n = 2500$ under the Plancherel--Hurwitz measure $\P_{n,\ell}^+$ in the high genus regime $\ell = 2 \lfloor 1.5 n \rfloor$ (sampled via a Metropolis--Hastings algorithm). The twofold asymptotic behaviour is shown in yellow: the first part $\boldsymbol{\lambda}_1$ is asymptotic to $\frac{2\ell}{\log(n)}$ and escapes the picture, while the rest of the partition $\tilde{\boldsymbol{\lambda}} = (\boldsymbol{\lambda}_2,\boldsymbol{\lambda}_3,\ldots)$ 
    has a VKLS limit shape at a scale of $\sqrt{n}$.
    See Theorem~\ref{thm:limitshape}. The profile of the partition is shown in red.
    }
    \label{fig:random_partition_hg_ls}\label{fig:main}
    
\end{figure}
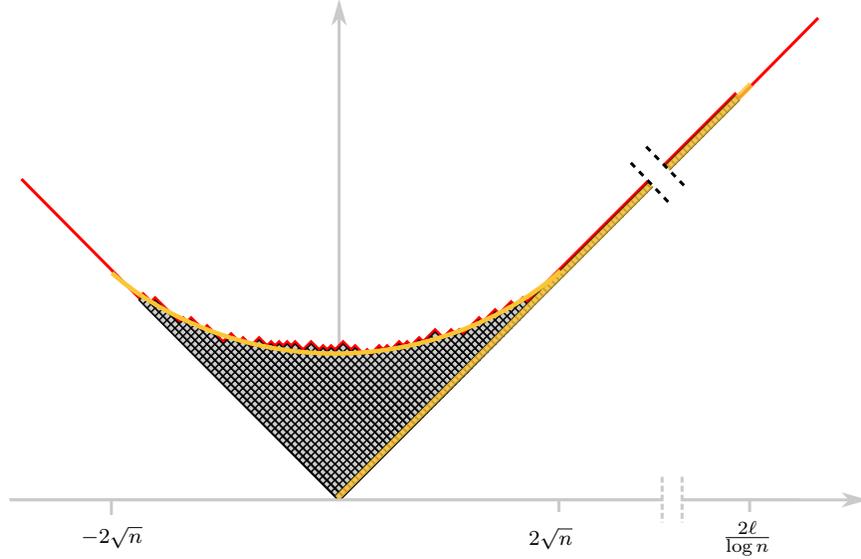

In this paper we will be interested in a generalisation of this measure, motivated by the study of Hurwitz numbers and branched coverings of the sphere, which are realized combinatorially by transposition factorisations or equivalently by certain combinatorial objects called Hurwitz maps, discussed in Section~\ref{sec:maps_intro}. We will study the asymptotic behaviour of this measure in a regime related to ``high genus'' surfaces, and show that this behaviour generalises the VKLS one in a novel way, see Figure~\ref{fig:main}.
For an even integer $\ell\geq 0$, we let $H_{n,\ell}$ be the number of factorisations of the identity of $\mathfrak{S}_n$ into $\ell$ transpositions:
\begin{equation}\label{eq:defHnl}
	H_{n,\ell} = \#\{ (\tau_1,\tau_2,\dots,\tau_\ell) \in (\mathfrak{S}_n)^\ell, \tau_1  \cdot \tau_2 \cdots \tau_\ell = \id , \mbox{ each } \tau_i 
	\text{ is a transposition} \}. 
\end{equation}
By a classical correspondence between transposition factorisations and branched coverings of the sphere going back to Hurwitz himself (\cite{Hurwitz_1891}, see also e.g.~\cite{LZ}), the number $H_{n,\ell}$ enumerates degree $n$ coverings of the Riemann sphere with $\ell$ numbered, simple, ramification points, by an oriented surface (connected or not). The number $H_{n,\ell}$ is called an \emph{unconnected Hurwitz number} in the enumerative geometry literature.
The \emph{Frobenius formula} from the representation theory of finite groups (see e.g.~\cite{Fulton_1996}) together with the combinatorial representation theory of $\mathfrak{S}_n$ provides an explicit expression of this number as a sum over partitions, vastly generalising~\eqref{eq:RSK} (which corresponds to  $\ell=0$). Indeed,
\begin{align}\label{eq:Frobenius}
H_{n,\ell} = 
	\frac{1}{n!}\sum_{\lambda \vdash n}{f_\lambda}^2  {C_\lambda}^\ell, 
\end{align}
where $C_\lambda$ is a combinatorial quantity, namely the sum of contents of all boxes of the partition $\lambda$ (see Section~\ref{sec:definitions}, or Figure~\ref{fig:young_diagrams2} directly), which can be expressed explicitly,
for a partition $\lambda=(\lambda_1,\lambda_2,\dots,\lambda_{\ell(\lambda)})$, as
\begin{equation} 
C_\lambda =  \sum_{i=1}^{\ell(\lambda)} \frac{\lambda_i (\lambda_i - 2i + 1) }{2}.
\end{equation}

The right hand side of the Frobenius formula~\eqref{eq:Frobenius} naturally gives rise to a certain measure on partitions, which is our main object of study:
\begin{definition}[Main object]\label{def:hurwitz_measure}
For $n\in \mathbb Z_{>0}$, $\ell \in 2 \mathbb Z_{\geq 0}$, the Plancherel--Hurwitz measure is the probability measure on partitions of $n$ defined by 
\begin{equation}
	\P_{n,\ell} (\lambda) := \frac{1}{ n! H_{n,\ell}} {f_\lambda}^2 {C_\lambda}^\ell.
\end{equation}
For $\ell>0$, the positive half of the Plancherel--Hurwitz measure is the probability measure on partitions of $n$ with positive content sum
\begin{equation}\P_{n,\ell}^+(\lambda) := \P_{n,\ell}(\lambda | C_\lambda > 0) = 2\cdot  \mathbf{1}_{C_\lambda>0} \cdot \P_{n,\ell}(\lambda)
.\end{equation}
\end{definition}
A partition distributed under $\P_{n,\ell}$ for $\ell>0$ can be thought of as a partition distributed under $\P_{n,\ell}^+$ whose Young diagram is reflected about a vertical axis with probability $\frac{1}{2}$.

When $\ell=0$ the measure $\P_{n,\ell}$ is nothing but the Plancherel measure.
Our main result deals instead with the case where $\ell$ grows linearly with $n$
(which we could call the \emph{high-genus regime} after the discussion of Section~\ref{sec:maps_intro} below).
The \emph{rescaled profile} $\psi_\lambda$ of a partition $\lambda$ is the piecewise linear function that encodes the countour of the tilted diagram of the partition, with coordinate axes scaled by a factor of $1/\sqrt{n}$, as shown in Figure~\ref{fig:main} -- see also Section~\ref{sec:definitions} for a formal definition or Figure~\ref{fig:young_diagrams1}. We have the following theorem (we denote random partitions with bold letters and fixed ones with plain letters throughout):

\begin{theorem}[Main result]\label{thm:limitshape}
Fix $\theta>0$ and let $\lh \vdash n $ be a random partition under the Plancherel--Hurwitz measure $\P^+_{n,\ell}$ in the	regime given by $\ell =\ell(n) \sim 2\theta n$. Then, as $n \to \infty$: 
\begin{longlist}[(i)]
 \item \label{thm:lst_first_part}the first part $\lh_1$ is equivalent to $\frac{2 \ell }{\log n}$ in probability, i.e. $\lh_1\left(\frac{2 \ell }{\log n}\right)^{-1}\xrightarrow{p}1$;
 \item \label{thm:lst_vkls_bulk}the rest of the partition $ \tilde{\lh} = (\lh_2, \lh_3, \ldots) $ has a VKLS limit shape. Namely, 
 \begin{align} 
 \sup_x| \psi_{\tilde\lh}(x) - \Omega(x)| \xrightarrow{p} 0,
 \label{eq:plancherel_limit_profile}
	 \mbox{ ~  with \;}\Omega(x) = \begin{cases}
 \frac{2}{\pi} \big(\arcsin \frac{x}{2} + \sqrt{4 - x^2}\big), &~ |x| \leq 2, \\
|x|, &~ |x|>2,
 \end{cases} 
 \end{align}
 where $\psi_{\tilde{\lh}} (x)$ is the rescaled profile of $\tilde{\lh}$;
 \item\label{thm:lst_2_sqrtn} we have $\frac{\lh_2}{\sqrt n}\xrightarrow{p}2$ and $\frac{\ell(\lh)}{\sqrt n}\xrightarrow{p}2$.
 \end{longlist} 
\end{theorem}

We could include  $\lh_1$ in the partition $\tilde{\lh}$ of~\eqref{eq:plancherel_limit_profile}, since the supremum norm in this statement is insensitive to a small number of large parts. However, from part~(\ref{thm:lst_2_sqrtn}) of the theorem, $\lh_1$ is the only part  not scaling as $\sqrt{n}$ so we find this formulation more natural. This limiting curve is graphed, along with the Young diagram of a finite sized partition sampled under a Plancherel--Hurwitz measure, in Figure~\ref{fig:main}.

Heuristically, we might think of this theorem as resulting from the two competing ``forces'' driving a random partition $\lh$ under  $\P^+_{n,\ell}$. On the one hand, due to the Plancherel factor ${f_{\lambda}}^2$, the classical VKLS theorem shows that there is a cost for the partition to deviate from the VKLS shape. On the other hand, the partition may prefer to deviate strongly from the VKLS shape if it gains a sufficiently high content-sum so that the factor ${C_{\lambda}}^\ell$ compensates the Plancherel loss.

Our result shows that a random partition balances these forces by obtaining a large content-sum exclusively from the first part $\lh_1$, and then leaving the rest of the partition to maximize only the Plancherel entropy. The different length scales determining the limit behaviour are a rather unique feature of this measure and its ``high genus'' regime; established generalisations of the Plancherel measure such as Schur measures and Schur processes do not exhibit such behaviour in typical asymptotic regimes~\cite{Okounkov_2001,Okounkov_Reshetikhin_2003}.

\subsection{High genus maps, Hurwitz numbers, and random walks} \label{sec:maps_intro}

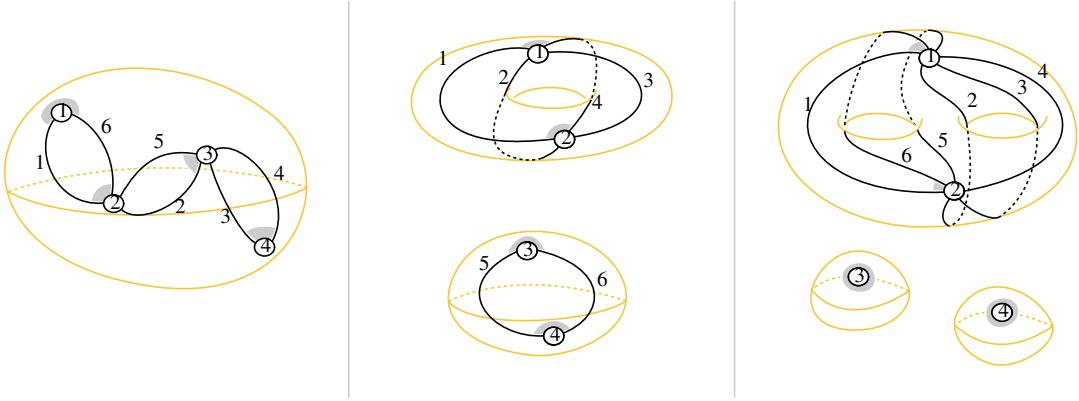
\begin{figure}
\centering
{\def\svgwidth{\sizeofpage}
\input{for_hurwitz_unconnected_selection.tex} 
}
\caption{Three pure Hurwitz maps, each with 4 vertices, 6 edges and Euler characteristic $\chi = 2$. Left, the  map corresponding to $(12)(23)(34)(34)(23)(12) = \id$ is connected and has genus 0; center, the map corresponding to $(12)^4(34)^2 = \id$ has two connected components, of genus 1 and 0; right, the map corresponding to $(12)^6 = \id$ has three components, of genus 2, 0 and 0. } \label{fig:unconnected_maps}
\end{figure}

Our original motivation to study the Plancherel--Hurwitz measure comes from the field of enumerative geometry and map enumeration.
A \emph{map} is a multigraph embedded on a compact oriented surface with simply connected faces, considered up to homeomorphisms. Equivalently it can be seen as a surface which is discretized by a finite number of polygons. Since the pioneering works of Tutte on planar maps (e.g.~\cite{Tutte_1963}) the enumeration of maps has proven to be particularly interesting, borrowing tools from physics, algebra and geometry and revealing their connections within combinatorics. These tools include matrix integral generating functions discovered by treating maps as Feynman diagrams~\cite{BIPZ_1978, LZ}, the topological recursion~\cite{Eynard_Orantin_2007}, recurrence formulas based on integrable hierarchies~\cite{Goulden_Jackson_2008, Carrell_Chapuy_2014}, classical generating functions (e.g.~\cite{Bender_Canfield_1986,BousquetMelouJehanne}) and bijective combinatorics~(e.g.~\cite{Schaeffer:PhD}). Such exact methods have led to the asymptotic enumeration of many types of maps on the plane or on surfaces of fixed genus $g$, which notably exhibit a universal counting exponent of $\displaystyle{\tfrac{5}{2}(g-1)}$ (e.g.~\cite{Chapuy:CPC2009}). 

 These methods do not, however, extend to maps whose genus grows quickly (in particular, linearly) with the number of polygons. 
This ``high genus'' regime, whose study was pioneered in~\cite{AngelChapuyCurienRay, Ray:diameter} in the case of one-face maps, is one of the most recent and exciting frontiers in the field, due the inefficiency of existing generating-function or bijective methods, requiring the development of new tools.

An advance in this direction was made recently by Budzinski and the second author~\cite{Budzinski_Louf_2021, Budzinski_Louf_2020}, who, as a byproduct of their work on the Benjamini--Curien conjecture~\cite{Curien_2016}, showed the following estimate for the number of (connected) triangulations of size $n$ on a surface of genus $g\sim \theta n$, by a combination of algebraic, combinatorial, and probabilistic methods:
\begin{equation} \label{eq:BL}
	T_{n,g} = n^{2g} \exp[c( \theta) n + o(n)], \ \ g\sim \theta n,
\end{equation}
where $c(\theta)>0$ is a known continuous function.
In this paper we will be interested in a different model of maps:
\begin{definition}[Hurwitz map] \label{def:hurwitz_map}
	A Hurwitz map with $n$ vertices and $\ell$ edges is a map on a (not necessarily connected) compact oriented surface, with vertices labelled from $1$ to $n$ and edges labelled from $1$ to $\ell$, which is such that the label of edges around each vertex increase (cyclically) counterclockwise. In such a map each vertex is incident to precisely one corner which is an  edge-label descent. If moreover each face of the map contains precisely one such corner, the Hurwitz map is called \emph{pure}.
\end{definition}
It is classical, and easy to see, that Hurwitz maps of parameters $n$ and $\ell$ are in bijection with tuples of transpositions $(\tau_1,\dots,\tau_\ell)$ in $\mathfrak{S}_n$, while \emph{pure} Hurwitz maps are in bijection with tuples whose product is equal to the identity. The bijection consists of identifying transpositions with edges of the map, and their index with the edge-label, as illustrated in Figure~\ref{fig:unconnected_maps} (this construction is a special case and an adaptation of the classical construction of ``constellations'', see~\cite{Bousquet-Melou_Schaeffer_2000,Duchi_Poulalhon_Schaeffer_2014,Chapuy:CPC2009}). 

In geometric language, pure Hurwitz maps are therefore in bijection with simply ramified, $n$-sheeted, branched covers of the sphere by an orientable surface, with $\ell$ numbered simple ramification points, and trivial ramification above $\infty$, which is the model considered in~\cite{Hurwitz_1891,Okounkov_2000,Dubrovin_Yang_Zagier_2017}.
Pure Hurwitz maps have also been studied from the combinatorial and probabilistic viewpoint, and they are known~\cite{Duchi_Poulalhon_Schaeffer_2014}, in the planar and fixed-genus cases, to belong to the same universality class as other natural models of maps such as triangulations, quadrangulations, etc. (the convergence to Brownian surfaces is only conjectured, but other properties of the universality class such as counting exponents or the existence of bijections are known).

It is important to insist that our surfaces (or maps) are not necessarily connected, which is a significant difference from  most models in the literature. A pure Hurwitz map of parameters $n$ and $\ell$ necessarily has $n$ faces, and its Euler characteristic $\chi$, its number of components $\kappa$, and its generalised genus $G$ (sum of the genera, or number of handles, of each connected component) are related by Euler's formula:
\begin{equation} \label{eq:euler_characteristic}
\chi = \#\text{vertices} - \#\text{edges} + \#\text{faces} = 2n - \ell = 2\kappa - 2G.
\end{equation}
For this reason we call regimes where $\ell\gg 2n$ ``high genus'' regimes. 

By the above correspondence, the number $H_{n,\ell}$ introduced in~\eqref{eq:defHnl} is the number of pure Hurwitz maps with $n$ vertices and $\ell$ edges. 
As a consequence of our analysis of the Plancherel--Hurwitz measure,  we obtain the following estimate
\begin{theorem}[Asymptotics of high genus unconnected Hurwitz numbers]\label{thm:asymptotic_H_nl} 
	As in~\eqref{eq:defHnl}, let $H_{n,\ell}$ be the unconnected Hurwitz number counting not necessarily connected pure Hurwitz maps with $n$ vertices and an even number $\ell= \ell(n) \sim 2\theta n $ of edges. 
	 Then, as $n\to \infty$,
\begin{equation}\label{eq:asymptoticsHnl}
	H_{n,\ell} = \bigg( \frac{\ell}{\log n}\bigg)^{2 \ell} \exp \big[ ( -2+ \log 2) \ell+o(n) \big],
\end{equation}
where the little-o is uniform for $\ell/n$ in any compact subset of $(0,\infty)$.
\end{theorem}

 It is tempting to see this theorem as as strong (for our model) as the Budzinski--Louf estimate~\eqref{eq:BL}, but unfortunately this is not quite the case. The major difference is that our maps are not necessarily connected. Indeed, we will show that even when $\theta \geq 1$ and there are sufficiently many edges to connect all the vertices we are predominantly counting unconnected maps,
in particular that 
\begin{theorem} \label{thm:nogiantcomp}
For all $\theta>1$, with high probability  (w.h.p., that is to say with probability tending to $1$ as $n\to\infty$) a uniformly random pure Hurwitz map with $n$ vertices and an even number $\ell = \ell(n) \sim 2  \theta n $ of edges contains a connected component with at least $\gamma(\theta) \ell$ edges, where $\gamma(\theta) = 2^{2\theta-1}-1$. However, in such a map, the largest connected component has (vertex) size $O_p( n /\log n) $, where $O_p$ denotes a big-O in probability.
\end{theorem}

The fact that the ``giant'' edge-component in the previous theorem has a sublinear number of vertices implies that its genus, viewed as a function of its number of vertices, is superlinear. This seems to rule out the possibility of deducing asymptotics for the \emph{connected} linear-genus regime from our results, at least not without new ideas.

At this point it is worth commenting that in map enumeration, the regime in which the genus is unconstrained, or superlinear, is often much easier to deal with than the linear case. 
In fact, the Plancherel--Hurwitz measure already appears (with no name) in the regime $\ell > \frac{1}{2} n \log n$, in work of Diaconis and Shahshahani on the transposition random walk on $\mathfrak{S}_n$.
They famously showed~\cite{Diaconis_Shahshahani_1981} that when $\ell \geq \frac{1+\epsilon}{2}n \log n$, the walk is asymptotically mixed after $\ell$ steps, and the proof essentially consists in showing that the Plancherel--Hurwitz measure is dominated by the trivial partition $(n)$ in this regime. 
In this context, our result~\eqref{eq:asymptoticsHnl} can also be interpreted as an estimate on the return probability of the random walk after linearly many steps -- much before the cut-off time, at a time when the Plancherel--Hurwitz measure still has a more subtle behaviour than the trivial partition.
\begin{corr}[Equivalent formulation of Theorem~\ref{thm:asymptotic_H_nl}]
	The probability that the transposition random walk on $\mathfrak{S}_n$ returns to the origin after $\ell= \ell(n) \sim 2 \theta n$ steps is equal to
	\begin{align}	
		\binom{n}{2}^{-\ell}H_{n, \ell}
		&= 	\exp \big[-2\ell (\log \log(n) - \log \theta  -2\log 2 + 1) + o(n)\big].
\end{align}
\end{corr}

Finally, we note that other measures on partitions have been studied by Biane in the context of the asymptotic factorisation of characters of $\mathfrak{S}_n$~\cite{Biane_2001}. The limit-shape phenomena observed in this reference are qualitatively different from ours, and possible connections are left to future work.

\subsection{Open questions and perspectives}

Maybe the main open question that follows our work is the following: \textit{does there exist an analogue of the RSK algorithm which proves the identity~\eqref{eq:Frobenius} combinatorially?} If this is the case, then our results about $\lambda_1$ and $\lambda_2$ probably translate into distributional limit theorems for certain parameters of random factorisations (or random pure Hurwitz maps). 
To start with, can one identify an interpretation of the statistic $\lambda_1$ in terms of factorisations of the identity by transpositions?

Another question is, of course, to know if one can use the Plancherel--Hurwitz approach to say anything about \emph{connected} Hurwitz maps of high genus. This would be very interesting. It may also be interesting to combine this approach with the technology of integrable hierarchies, which have been so fruitful but have so far not directly led to precise asymptotic estimates nor limit theorems for connected random maps or Hurwitz numbers of high genus. In any case, obtaining the analogue of the Budzinski--Louf~\cite{Budzinski_Louf_2021} estimates for connected pure Hurwitz maps (that is, an exponential estimate up to the linear order, of the kind in~\eqref{eq:BL}) is an important question in the direction of completing the universality picture.

Another natural extension of the present work, which might be linked to the previous question, would be to consider other regimes for $\ell$, especially for sublinear values of $\ell$. It is natural to expect that when varying the parameter $\ell$ from $\ell=0$ to $\ell\sim 2\theta n$, an asymptotic statement interpolating between our main result and the VKLS theorem should hold, with a larger and larger first part $\lambda_1$. Heuristic calculations suggest a possible phase transition 
with markedly different behaviour for $\ell = O(\sqrt{n})$, which might be the threshold after which the contribution to the content-sum is overwhelmingly made by the first part. We leave these questions to further work.

Finally, among all the models of branched coverings/ Hurwitz numbers/ combinatorial maps, pure Hurwitz maps is the one for which the connection to the Plancherel measure is the most combinatorial, which made it the most natural candidate to test the idea of using random partition techniques to study the high genus regime. However it would be interesting to study other deformations of the Plancherel measure motivated by Hurwitz number theory, for example in the context of weighted Hurwitz numbers of~\cite{GPH:weightedHurwitz, ACEH}.

\subsection{Plan of the paper}

In Section~\ref{sec:partitions}, we provide full definitions and context for 
the Plancherel--Hurwitz measure, we overview the VKLS theorem and related estimates for the classical Plancherel case, and we derive new deviation bounds for the sum of contents under Plancherel measure, that we will need in Section~\ref{sec:microproofs}. The remaining sections are dedicated to proofs. The proof of Theorem~\ref{thm:limitshape} is divided into two sections, for the two scaling regimes associated with the limit shape. In Section~\ref{sec:macroproofs}, we focus on estimating the order of the first part $\lh_1$, thus proving part (\ref{thm:lst_first_part}) of Theorem~\ref{thm:limitshape} and by extension part~(\ref{thm:lst_vkls_bulk}) of the same theorem and Theorem~\ref{thm:asymptotic_H_nl}. In Section~\ref{sec:microproofs}, we will study the finer details of the limit shape and control the second part $\lh_2$ and the number of parts $\ell(\lh)$ to prove Theorem~\ref{thm:limitshape}, part~(\ref{thm:lst_2_sqrtn}). Finally in Section~\ref{sec:mapsproofs} we will study the properties of connected components in pure Hurwitz maps and prove Theorem~\ref{thm:nogiantcomp}.

\section{Partitions and the Plancherel measure}
\label{sec:partitions}

\subsection{Preliminary definitions for partitions} \label{sec:definitions}
Here, we gather definitions for the notions used to define the Plancherel--Hurwitz measure (Definition~\ref{def:hurwitz_measure}) and to describe a partition's shape, as used in Theorem~\ref{thm:limitshape}.

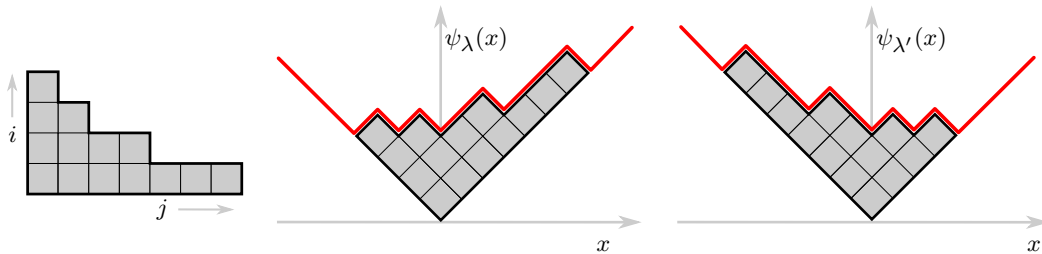
\begin{figure}[t]
\centering
{\def\svgwidth{0.22\sizeofpage}
\input{for-yd-french.tex} \hfill \def\svgwidth{0.342\sizeofpage}  \input{for-yd-russian.tex} \hfill \def\svgwidth{0.37\sizeofpage} \input{for-yd-conjugate.tex} }
    \caption{Left, the Young diagram of the partition $\lambda = (7,4,2,1) \vdash 14$ in the French convention; center, the diagram of $\lambda$ in the Russian convention, with the profile $\psi_\lambda$ indicated in red (coordinate axes are in grey); right, the diagram and profile of the conjugate partition $\lambda^\prime=(4,3,2,2,1,1,1)$. } \label{fig:young_diagrams1}
\end{figure}

Formally, a partition $\lambda  = (\lambda_1 \geq \lambda_2 \geq  \ldots \geq \lambda_{\ell(\lambda)})$ of $n$ is a weakly decreasing finite sequence of $\ell(\lambda)$ positive integers (called parts) which sum to $n$. We denote this by ``$\lambda \vdash n$'', and call $n = |\lambda| := \sum_i \lambda_i$ the \emph{size} of $\lambda$ and $\ell(\lambda)$ its \emph{length}, and for convenience append the sequence with $\lambda_i = 0$ for $i > \ell(\lambda)$. The \emph{conjugate partition} $\lambda^\prime$ of $\lambda$ is the partition with parts $\lambda_j^\prime = \#\{i|\lambda_i \geq j\}$. A partition may be represented by its \emph{Young diagram} (see Figure~\ref{fig:young_diagrams1}). In the French convention, the Young diagram of a partition $\lambda \vdash n$ is a stack of left aligned rows of boxes with $\lambda_i$ boxes in the $i$th row from the bottom; the number of boxes in the $j$th column is $\lambda_j^\prime$. We shall refer to Young diagrams in the Russian convention, which are obtained by rotating a French Young diagram counterclockwise by 45$^\circ$. Then, the Young diagram of the conjugate partition $\lambda^\prime$ is simply the Young diagram of $\lambda$ flipped about a vertical axis. We shall refer to partitions and their Young diagrams interchangeably throughout this text.

Adding a coordinate system to a diagram in the Russian convention with its origin centered at the bottom corner, a partition is encoded by its \emph{profile}, which is the piecewise linear function describing the upper edge of its diagram extended out to a line of slope $+1$ to the right and a line with slope $-1$ to the left. When $n$ is large, a partition of $n$ is best described by its \emph{rescaled profile} $y = \psi_\lambda(x)$, whose coordinates are scaled by $1/\sqrt{n}$ relative to the distance between the centers of boxes in 
its Young diagram; it is composed of elements with slope $-1$ defined implicitly by the equations 
\begin{equation} \label{eq:profile}
\psi_{\lambda}(x) = \frac{\lambda_{\lfloor u \rfloor+1}+u}{\sqrt{n}}, \qquad x = \frac{ \lambda_{\lfloor u \rfloor+1} - u}{\sqrt{n}}; \qquad  u \in (0, \infty). 
\end{equation}
and elements with slope $+1$ defined implicitly by
\begin{equation}
\psi_{\lambda}(x) = \frac{\lambda^\prime_{\lfloor v \rfloor+1}+v}{\sqrt{n}}, \qquad x = \frac{ v-\lambda^\prime_{\lfloor v \rfloor+1} }{\sqrt{n}}; \qquad  v \in (0, \infty). 
\end{equation}
Note that $\psi_\lambda(x)$ is 1-Lipschitz.

In what follows we will repeatedly use the following bound for the number of integer partitions of $n$,
\begin{equation} \label{eq:asymp_number_partitions}
	\#\{\lambda \vdash n\} = \exp \big[ O(\sqrt{n}) \big],
\end{equation}
which follows for example from the precise asymptotic estimate of  Hardy and Ramanujan~\cite{HardyRamanujan}.

A \emph{Standard Young Tableau (SYT)} of shape $\lambda\vdash n$ is a filling of the boxes of the Young diagram of $\lambda$ with all the numbers from $1$ to $n$ such that they increase along rows and columns.
The number $f_\lambda$ of such tableaux can be calculated by the \emph{hook-length formula} 
\begin{equation} \label{eq:hooklength}
f_\lambda = \frac{ n! }{\prod_{\square \in \lambda} \eta_{\lambda}(\square)}, \qquad \eta_\lambda(\square_{i,j}) = \lambda_i - i + \lambda^\prime_j - j + 1,
\end{equation}
where the product is taken over all boxes of the diagram of $\lambda$, and the hook length $\eta_\lambda(\square)$ is the number of boxes in a hook connecting $\square$ to the edge of the diagram of $\lambda$, written here explicitly for the box $\square_{i,j}$ in the $i$th row and the $j$th column (see Figure~\ref{fig:young_diagrams2}). We have, quite immediately, $f_{\lambda} = f_{\lambda^\prime}$.

The \emph{content} $c(\square)$ of a box $\square$ gives the horizontal position of the center of the box in its diagram in the Russian convention: for the box $\square_{i,j}$, it is defined  by $c(\square_{i,j}) = j - i$. Adding the contents of all the boxes in the Young diagram of $\lambda$, its \emph{content-sum} is given by 
\begin{equation} \label{eq:contentsum_def}
C_\lambda := \sum_{\square \in \lambda} c(\square) = \sum_{i=1}^{\ell(\lambda)} \frac{\lambda_i (\lambda_i - 2i + 1) }{2}.
\end{equation}
It is straightforward to see that $C_{\lambda^\prime} = - C_{\lambda}$, and that if we exclude the partitions with $C_\lambda = 0$,
the sign of content-sum provides a natural way to separate the remaining partitions of $n$ into positive and negative halves, which can be obtained from one another by conjugation -- hence the introduction of the positive half measure $\P^+_{n,\ell}$.

\begin{figure}[t]
\centering
{\def\svgwidth{0.3\sizeofpage}
\input{for-frenchyd-syt.tex} \hfill \def\svgwidth{0.3\sizeofpage}  \input{for-frenchyd-hooks.tex}\hfill  \def\svgwidth{0.3\sizeofpage} \input{for-yd-csum.tex} }
    \caption{The Young diagram of the partition $(4,2,1) \vdash 7$ (in the Russian convention). Left, its boxes are filled to produce a SYT of shape $(4,2,1)$; center, they are filled with their hook lengths, showing there are $f_{(4,2,1)} = 7!/(6\cdot 4 \cdot 3 \cdot 2) = 35$ such tableaux; right, each box is filled with its content, and the content-sum is $C_{(4,2,1)} = -2-1+0+0+1+2+3=3$.} \label{fig:young_diagrams2}
\end{figure}
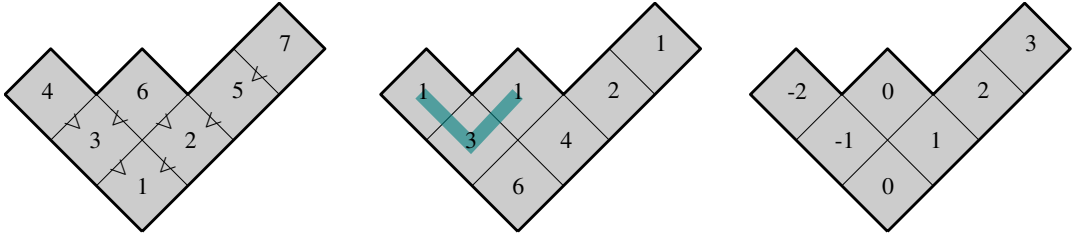

\medskip
We  conclude this subsection by saying a word on the representation theoretic origin of the connection (Equation~\eqref{eq:Frobenius}) between factorisations and the combinatorial quantities $f_\lambda$, $C_\lambda$ just defined.
First, the number $f_\lambda$ of SYT of shape $\lambda \vdash n$ is famously equal to the dimension of the irreducible representation $V^\lambda$ indexed by $\lambda$ of the symmetric group $\mathfrak{S}_n$ of order $n$, see for example \cite{Fulton_1996}. 
Because the set of all transpositions forms a conjugacy class in $\mathfrak{S}_n$, Frobenius's formula from the representation theory of finite groups (e.g. \cite[Thm A.1.9]{LZ}) directly implies that the number of factorisations of the identity into $\ell$ transpositions can be written as:
\begin{align}
H_{n,\ell} = 
	\frac{1}{n!}\sum_{\lambda \vdash n}{f_\lambda}^2  (\tilde \chi^\lambda (R))^\ell, 
\end{align}
where $\tilde \chi^\lambda (R)$ is the scalar (normalized character) which describes the action of the sum of all transpositions on the module $V^\lambda$ (the existence of this scalar is a consequence of Schur's lemma). The fact that $\tilde \chi^\lambda (R)$ is equal to the sum of contents $C_\lambda$ is a classical fact, which is most naturally understood via the connection between contents and Jucys--Murphy elements, see e.g.~\cite{Murphy1981}.

\subsection{Key elements from the classical Plancherel case}

For $\ell=0$ the Plancherel--Hurwitz measure becomes the Plancherel measure $\P_{n,0}(\lambda) = \frac{1}{n!} f_\lambda^2$ which, as said in the introduction, is very well understood.
\begin{theorem}[VKLS limit shape theorem~\cite{Logan_Shepp_1977,Vershik_Kerov_1977}]\label{thm:vkls}
Let $\lh \vdash n $ be a random partition under the Plancherel measure $\P_{n,0}$. Then, as $n\to \infty$ we have
 \begin{equation} \label{eq:strong_vkls}
 \sup_x | \psi_{\lh} (x) - \Omega(x)| \xrightarrow{p} 0 \qquad \text{and} \qquad \frac{\lh_1}{\sqrt n} \xrightarrow{p} 2, \; \qquad \frac{\ell(\lh)}{\sqrt n} \xrightarrow{p} 2 
 \end{equation}
 where $\psi_{\lh} (x)$ is the rescaled profile of $\lh$ and  $\Omega(x)$ is the curve defined at~\eqref{eq:plancherel_limit_profile}.
\end{theorem}
Several proofs exist for this limit shape result. Perhaps the simplest and most conceptual ones use the formulation of the Plancherel measure in the language of fermions and the infinite wedge space, which provides a direct connection with determinantal processes~\cite{Borodin_Okounkov_Olshanski_2000,Johansson_1998}.
Such approaches and their generalisations have grown into a vast field of research after the introduction of the theory of Schur processes (see e.g.~\cite{generalSchurProcessRef} for an entry point).

In the case $\ell>0$ that we study here, it is still possible (and natural) to formulate the Plancherel--Hurwitz measure in terms of the infinite wedge, see~\cite{Okounkov_2000}. This leads to a deep connection with integrable hierachies (the KP and 2-Toda hierarchies in particular), and even to a simple looking recurrence formula to compute the numbers $H_{n,\ell}$ (more precisely, their connected counterpart, see~\cite{Okounkov_2000,Dubrovin_Yang_Zagier_2017}). However, we do not know how to use either of these tools to approach our problems.

\smallskip

Other, maybe more elementary, proofs of the VKLS theorem are based on a direct scaling of the hook-length formula~\eqref{eq:hooklength} and variational calculus.
We recommend the first chapters of the book~\cite{Romik_2015} as a useful reference for such approaches. A key outcome of such an approach is the following estimate for the Plancherel measure  of a partition $\lambda$ in terms of its rescaled profile $\psi_\lambda$, see e.g.~\cite[Section 1.14]{Romik_2015}
\begin{equation} \label{eq:plancherel_integral}
\P_{n,0}(\lambda) = \frac{1}{n!} f^2_\lambda = \exp \bigg[ -n\bigg(1 + 2I_{\text{hook}}(\psi_\lambda) + O\bigg( \frac{\log n}{\sqrt{n}} \bigg) \bigg) \bigg]
\end{equation}
where $I_{\text{hook}}(\cdot)$ is an ``energy'' functional defined by the integral formula
\begin{equation}
    I_\text{hook} (\psi_\lambda) = \int_0^\infty \int_0^{\phi(u)} \log  \eta_{\psi_\lambda} (u,v)\, dv \, du, \qquad \eta_{\psi_\lambda}(u,v) = \phi_\lambda(u) - u +  \phi_{\lambda^\prime}(v)-v
\end{equation}
where $\phi_\lambda(u) = \lambda_{\lfloor \sqrt{n} u \rfloor + 1}/\sqrt{n} $ and $\phi_{\lambda'}(v)$ is defined similarly, so that $\eta_{\psi_\lambda}$ is the rescaled hook-length. 
The VKLS function $\Omega(x)$ given explicitly  at~\eqref{eq:plancherel_limit_profile} is the unique continuous function satisfying $\int (\Omega(x)-|x|)dx =1$ which minimizes $I_{\text{hook}}(\cdot)$ (see e.g.~\cite[Section 1.17]{Romik_2015}).  This (and continuity properties of $I_{\text{hook}}$)
implies the limit shape part of the VKLS theorem, since any partition whose profile is ``far'' from $\Omega(x)$ will  appear with an exponentially small probability.

While the first part $\lh_1$ is bounded below by the convergence of $\psi_{\lh}$ to $\Omega$ in the supremum norm, proving Theorem~\ref{thm:vkls} also requires an upper bound on $\lh_1$ that does not directly follow from this limit shape analysis and requires an additional method. One approach is to consider the \emph{Plancherel growth process}, explicitly introduced by Kerov~\cite{Kerov1998} and presented pedagogically in~\cite[Section 1.19]{Romik_2015}.
This process uses the RSK algorithm to provide a coupling between random Plancherel partitions of size $n-1$ and $n$, via a natural coupling at the level of random permutations.
Our proof of part (iii) of Theorem~\ref{thm:limitshape} does not directly use the Plancherel growth process, but it still relies on a comparison of measures for $n-1$ and $n$. For this we will need the following identity, which comes from the fact that the growth process is a well normalized stochastic process.
\begin{prop}[Plancherel growth process normalisation~\cite{Kerov1998}] \label{prop:pgp_norm}
Let $\mu \vdash n-1$ and let $\mu\nearrow \nu$ denote that $\nu \vdash n$ is obtained from $\mu$ by adding one box. We have
\begin{equation}
\sum_{\nu:\mu\nearrow \nu} {f_\nu} = n f_\mu .
\end{equation}
\end{prop}

\subsection{Deviation bounds for the Plancherel measure}

In the literature, precise deviation bounds can be found for $\lh_1$ deviating from $2\sqrt n$ under the Plancherel measure. We state here one of them\footnote{the result in~\cite{DZ99} is actually more precise.} that will be useful for us in Section~\ref{sec:microproofs}:

\begin{prop}[\cite{DZ99}, equation (1.5)]\label{prop:DZ}
For every $\epsilon>0$, there exists $K>0$ such that 
\begin{equation}\label{eq_deviation_first_part_plancherel}
\P_{n,0}(\lh_1\geq (2+\epsilon)\sqrt n)\leq \exp[-(1+o(1))K\sqrt n].
\end{equation}
\end{prop}

We will now prove a deviation bound for the content sum of a random partition under the Plancherel measure. To the best of our knowledge, such a  bound is new, and we think it might be of independent interest. However, we do not think this bound is optimal, nor that conditions on  $\lambda_1$ and $\ell(\lambda)$ are necessary -- finding the best upper bound on these deviation probabilities seems an interesting problem that we leave open.

\begin{theorem}[Deviation bound for the sum of contents under the Plancherel measure]\label{thm:deviationContents}
For every constant $c\geq 2$, there exists a constant $B$ such that when $t \left(\frac{\log n}{\sqrt n}\right)^{-1/6}$ goes to zero, one has
\begin{equation}\label{eq_deviation_content}
\P_{n,0}\left(C_{\lh}\geq n^{3/2}t\text{ and }\max(\lh_1,\ell(\lh_1))\leq c\sqrt n\right)\leq \exp[-(B+o(1))nt^{6}].
\end{equation}
\end{theorem}

\begin{proof}

Let us define the height function of the rescaled profile as $\hl(x):=\psi_\lambda(x)-|x|$, its distance above the lines $y = |x|$.
To any function $h$, we can associate its Fourier transform $\hat{h}$:
\begin{equation}
\widehat h(u)=\int_{-\infty}^{\infty}e^{-ixu}h(x)dx.
\end{equation}
We can calculate the content sum of $\lnn$ easily from $\hl$:

\begin{equation}\label{eq_content_int}
	C_\lambda=\frac{n^{3/2}}{\sqrt 2}\int_{-\infty}^{\infty}x\hl(x)dx=\frac{n^{3/2}}{\sqrt 2} {\widehat {\hl}}'(0).
\end{equation}

In \cite[equation 1.19]{Romik_2015}, it is shown that for all $\lnn$
\begin{equation}\label{eq_hook_functional}
\frac{f_\lambda^2}{n!}=\exp[-n(J(\hl)-J(h_0))+O(\sqrt n \log n)].
\end{equation}
 where $h_0$ is the height function of the VKLS curve $\Omega$, and $J=2I_{\text{hook}}$. For our proof, we will only need the following inequality (\cite[equations (1.35) and (1.38)]{Romik_2015})
 \begin{equation}\label{ineq_J_Q}
 J(\hl)-J(h_0)\geq Q(\hl-h_0),
 \end{equation}
where 
\begin{equation}Q(h):=\frac 1 4\int_{-\infty}^{\infty} |u||\widehat h(u)|^2du.\end{equation}

Let $h_*:=\hl-h_0$ for $\lambda$ satisfying $\max(\lone,\ll)\leq c\sqrt n$. Then the support of  $\hl$ is contained in $[-c,c]$, and thus so is the support of $h_*$ (because $h_0$ is supported on $[-2,2]$).
We have 
\begin{align}
	|{\widehat{h_*}}'(u)-{\widehat{h_*}}'(0)|=\left\lvert\int_{-\infty}^{\infty}(e^{ixu}-1)xh_*(x)dx\right\rvert 
 \end{align}
 and applying the inequality $|e^{it}-1]\leq |t|$ this gives 
 \begin{align}
|{\widehat{h_*}}'(u)-{\widehat{h_*}}'(0)| &\leq \int_{-c}^{c} x^2|u||h_*(x)|dx\nonumber\\
&\leq |u| \int_{-c}^{c} c^2 (\hl(x)+h_0(x))dx = 2c^2|u|,
\end{align}
where we use the fact that $h_0\geq 0$ and $h_\lambda\geq 0$, and obtain the last expression for the upper bound from the fact that $\int_{-c}^{c}\hl=\int_{-c}^{c}h_0=1$.
Now, for all $v>0$, we have
\begin{align}
|{\widehat{h_*}}(v)-v{\widehat{h_*}}'(0)|&\leq \int_{0}^{v}|{\widehat{h_*}}'(u)-{\widehat{h_*}}'(0)|du\leq \int_{0}^{v} 2c^2udu
=C v^2
\end{align}
for the constant $C = \frac{2c^2}{3}>0$.
Therefore by the triangular inequality, we have
\begin{equation}|{\widehat{h_*}}(v)|\geq v|{\widehat{h_*}}'(0)|-C v^2. \end{equation}
This gives 
\begin{align}
Q(\hl-h_0)=\frac 1 4\int_{-\infty}^\infty |v||{\widehat{h_*}}(v)|^2dv&\geq \frac 1 4\int_0^\epsilon v|{\widehat{h_*}}(v)|^2dv\notag\\
&\geq\frac 1 4 \int_0^\epsilon v\left( v|{\widehat{h_*}}'(0)|-C v^2\right)^2 dv
\end{align}
for any $\epsilon >0$. Now we set $\epsilon :=\frac {|{\widehat{h_*}}'(0)|} {2C}$, and note that for all $v\in[0,\epsilon]$ one has $v|{\widehat{h_*}}'(0)|-C v^2\geq \frac{v|{\widehat{h_*}}'(0)|}{2}\geq 0$. Inserting this into the lower bound above, we obtain
\begin{align}\label{eq_bound_Q}
Q(\hl-h_0)\geq \frac 1 4\frac{|{\widehat{h_*}}'(0)|^2}{4} \int_0^\epsilon v^3 dv = C' |{\widehat{h_*}}'(0)|^6
\end{align}
for the constant $C'= \frac{1}{4^{5} C^{4}}>0$.

	Now, if $\Cl\geq n^{3/2}t$, by~\eqref{eq_content_int} we have  ${\widehat {h_*}}'(0)={\widehat\hl}'(0)\geq \sqrt 2 t$, which by~\eqref{eq_bound_Q} implies
\begin{equation}Q(\hl-h_0)\geq 8C' t^6, \end{equation}
which by~\eqref{eq_hook_functional} and~\eqref{ineq_J_Q} implies the result.
\end{proof}

\section{Proofs: Macroscopic features} \label{sec:macroproofs}

In this section we will prove parts~(\ref{thm:lst_first_part}) and (\ref{thm:lst_vkls_bulk}) of Theorem~\ref{thm:limitshape}, which describe the first part and the general shape of a random partition under the measure $\P_{n,\ell}^+$ where $\ell\sim 2\theta n$. As we will see, to establish these ``macroscopic'' characteristics we will show that the cost to the measure associated with deviating from the typical behaviour is exponential (this will no longer be true in Section~\ref{sec:microproofs}). We will also prove Theorem~\ref{thm:asymptotic_H_nl}, the approximate asymptotics for $H_{n,\ell}$ at high genus, using the macroscopic limit shape and the intermediate results used to prove it. 

\subsection{Notation}

For any set $\Lambda $ of partitions of $n$ we write
\begin{equation}\Z_n(\Lambda):=\frac{1}{n!}\sum_{\lambda\in \Lambda} f_\lambda^2 (C_\lambda)^{\ell}.\end{equation}
Note that the  partition function of our model is  $H_{n,\ell} = \Z_n(\{\lambda \vdash n\})=\frac{1}{n!}\sum_{\lambda\vdash n} f_\lambda^2 (C_\lambda)^{\ell}$.

We also fix $\varepsilon := \frac{1}{100}$, and we split any partition $\lambda \vdash n$ into 
\begin{equation}\lambda = \lambda^{+} \sqcup \lambda^{-},\end{equation} 
where $\lambda^{+}$ denotes the parts of $\lambda$ that are greater than $n^{1-\varepsilon}$ and $\lambda^{-}$ the parts that are less than $n^{1-\varepsilon}$, see Figure~\ref{fig:split_partition}. 
The value of $\varepsilon$ is somewhat arbitrary at this stage, and will not affect our final results. With this threshold in mind we establish some convenient sets. For $M \in [n^{1-\varepsilon},n]$ and given $\mu \vdash n - M$ we let
\begin{equation}\Lambda(\mu,M):=\{\lambda||\lp|=M, \lm=\mu\}.\end{equation}
Then for $m \in [0,M]$ we set 
\begin{equation}\label{eq:Lambdadef}\Lambda(\mu,M,m):=\{\lambda\in \Lambda(\mu,M)|\lone=M-m\}.\end{equation}
We also introduce the notation $\lambda^0= M \sqcup \mu$, such that 
$\Lambda(\mu,M,0)=\{\lambda^0\}$. 

 In addition, we introduce the following set of partitions (which depends implicitly on the integers~$n$ and~$\ell$):
\begin{equation}\label{eq:Lambdastardef} \Lambda^*:=\left\{\lnn \ \middle|\ \lambda^+=(\lambda_1)\text{ and }\lambda_1\in \left[\frac{\mt \ell}{\log n}, \frac{\Mt \ell}{\log n}\right]\right\}.\end{equation}

\medskip
Finally, we establish the following convention:
{\it from now on (Sections~\ref{sec:macroproofs}--\ref{sec:microproofs}--\ref{sec:mapsproofs}), all little-o's and big-O's are uniform for $\ell/n$ in any compact subset of $(0,\infty)$ (in addition to uniformity in other quantities, which is stated as appropriate).}

\subsection{Deterministic estimates and lower bound on \texorpdfstring{$H_{n,\ell}$}{Hnl}}

We will use the following convenient bounds.
\begin{lemma}[Useful bounds]\label{lemma:factss}
Let $\lambda\vdash n$ with $\lambda^+=(\lambda_1,\dots,\lambda_p)$, then
\begin{longlist}[(i)]
\item\label{f1} $\displaystyle\frac{1}{n!}f_\lambda^2\leq \frac{n!}{\prod_{i=1}^p (\lambda_i!)^2(n-|\lp|)!)^2}f_{\lm}^2 \leq \frac{n!}{{\prod_{i=1}^p (\lambda_i!)^2}(n-|\lp|)!}\leq \frac{n^{|\lp|}}{\prod_{i=1}^p (\lambda_i!)^2}$,
\item \label{f2}$C_\lambda\leq \dfrac{\lone n}{2}$,
\item\label{f3} $C_\lambda=C_{\lambda^+} - p|\lambda^-| + C_{\lambda^-} = \displaystyle\sum_{i=1}^p\frac{\lambda_i(\lambda_i-2i + 1)}{2}-p|\lm|+C_{\lm}$,
\item\label{f4} $C_\lambda\leq \frac{|\lp|^2}2+\frac{n^{2-\eps}}{2}$.
\end{longlist}
\end{lemma}

\begin{proof}
\noindent \textit{(i).}
One can fill a Young diagram of shape $\lambda\vdash n$ with distinct numbers from $1$ to $n$ by picking $\lambda_1$ numbers then filling the first row with them in increasing order, then doing the same for the second row with $\lambda_2$ numbers and so on until the $p$th row. There are at most $\binom{n}{\lambda_1,\lambda_2,\ldots,\lambda_p}$ ways to do so, and once this is done there are at most $f_\lm$ ways to fill the remaining rows.
Moreover, from~\eqref{eq:RSK} we have $f_{\lm}^2\leq |\lm|! = (n-|\lp|)!$. Therefore we obtain
\begin{equation}
\frac{1}{n!}f_\lambda^2\leq \frac{n!}{\prod_{i=1}^p (\lambda_i!)^2(n-|\lp|)!)^2}f_{\lm}^2 \leq
\frac{n!}{{\prod_{i=1}^p (\lambda_i!)^2}(n-|\lp|)!}.    
\end{equation}
The last inequality of the claim is straightforward.

\noindent \textit{(ii).} From the definition of $C_\lambda$, given by~\eqref{eq:contentsum_def},
\begin{equation}C_\lambda=
\sum_{i=1}^{\ell(\lambda)} \frac{\lambda_i (\lambda_i - 2i + 1) }{2}
\leq 
\sum_{i=1}^{\ell(\lambda)} \frac{\lambda_i \cdot \lambda_i }{2} \leq 
	\frac{\lambda_1}{2} \sum_{i=1}^{\ell(\lambda)} \lambda_i = \dfrac{\lone n}{2}.\end{equation}

\noindent \textit{(iii).} Splitting~\eqref{eq:contentsum_def} into contributions from the first $p$ rows and the subsequent ones, we see that the former just contribute $C_{\lambda^+}$ to $C_\lambda$,  while for the latter the content of each box is the content of a box of $\lambda^-$ shifted by $-p$ (the number of parts in $\lambda^+$). This gives the result.

\noindent \textit{(iv).} By (ii), $C_\lm\leq \frac{n^{2-\eps}}{2}$ since all parts of $\lm$ are smaller than $n^{1-\eps}$. Applying this to (iii) yields the result.
\end{proof}

\begin{lemma}[Bounding the normalisation from below]\label{lem_lowerbound_Z_n}
We have
\begin{equation}\label{eq_lowerbound_Zn_rough}
H_{n,\ell} \geq \left(\frac{\ell}{\log n}\right)^{2\ell} \exp\left[-(2-\log2)\ell+O(\sqrt{n}\log^2 n)\right].
\end{equation}
\end{lemma}
\begin{proof}
	Let  $L:=\lfloor\frac{2\ell}{\log n}\rfloor$, 
	 and let $\mu$ be a partition of $n-L$ maximizing $f_\mu$ among partitions with $\mu_1\leq 3 \sqrt{n}$ and $\ell(\mu)\leq 3 \sqrt{n}$. We let $\lambda^*=L\sqcup \mu$. Using Lemma~\ref{lemma:factss}~(iii), and noting that $C_\mu \geq -\frac 3 2 n^{3/2}$ (by Lemma~\ref{lemma:factss}~(ii)), we have
	\begin{equation}\label{u1}
	C_{\lambda^*}=\frac{L(L+1)}{2}-|\mu|+C_{\mu}\geq \left(1+O\left(\frac{\log^2 n}{\sqrt{n}}\right)\right)\frac{L^2}2.
	\end{equation}
	On the other hand, consider the sum
	\begin{equation}\sum_{\nu\vdash n-L} f_\nu^2 = (n-L)!.
	\end{equation}
	By the VKLS theorem, we know that this sum is dominated by partitions such that $\nu_1\leq 3\sqrt{n}$. Since the number of terms is bounded by~\eqref{eq:asymp_number_partitions}, we deduce that 
\begin{equation}\label{eq_local_value_f_mu}
    f_\mu^2 = (n-L)!e^{O(\sqrt n)}.
\end{equation}

Therefore, we have
\begin{align}
Z_n(\{\lambda^*\})&=\frac{1}{n!}f_{\lambda^*}^2C_{\lambda^*}^\ell 
\geq \frac{1}{n!}f_{\mu}^2C_{\lambda^*}^\ell \notag\\
&\geq 
  \frac{(n-L)!e^{O(\sqrt n)}}{n!} C_{\lambda^*}^\ell  
   \notag\\
   &\geq n^{-L} e^{O(\sqrt n)}\bigg( \frac{L^2}{2}\bigg)^\ell\left(1+O\left(\frac{\log^2 n}{\sqrt{n}}\right)\right)^\ell \notag \\
&= \exp\left[2\ell(\log \ell-\log\log n)-\ell(2-\log2)+O(\sqrt{n}\log^2 n)\right].
\end{align}
We successively used the (in)equalities $f_{\lambda^*}\geq f_{\mu}$,~\eqref{eq_local_value_f_mu} and~\eqref{u1}.

This finishes the proof since $H_{n,\ell} =\Z_n(\{\lambda \vdash n\}) \geq \Z_n(\{\lambda^*\})$.
\end{proof}

\subsection{
First bound on the first and second parts}

We now proceed with a succession of lemmas that gradually give better control on the partition $\lh^+$. 
\begin{lemma}[Controlling big parts] \label{lem_control_big} Let $\lh$ be a random partition of $n$ under the assumptions of Theorem~\ref{thm:limitshape}. Then $|\lh^+| \in \left[\frac{\mt \ell}{\log n}, \frac{\Mt \ell}{\log n}\right]$ w.h.p.
\end{lemma} 

\begin{figure}
\centering
{\def\svgwidth{\sizeofpage}
\input{for_diagram_comparison} 
}
\caption{Left, a partition $\lambda \vdash n$ in $\Lambda(\mu,M,m)$, with $\lambda^+$ and $\lambda^-$ indicated. Center, a SYT of shape $\lambda^0 \in \Lambda(\mu,M,0)$ (the filling of the boxes is not shown) is transformed to a SYT of some shape $\lambda \in \Lambda(\mu,M,m) $ or to something else by the surjective operation used to prove Claim~\ref{fact_bound_g_f}. Right, a partition $\lambda = L\sqcup \nu^\prime$ with $|\nu^+| = m$, used to prove Lemma~\ref{lemma:lengthbd}.} \label{fig:split_partition}
\end{figure}
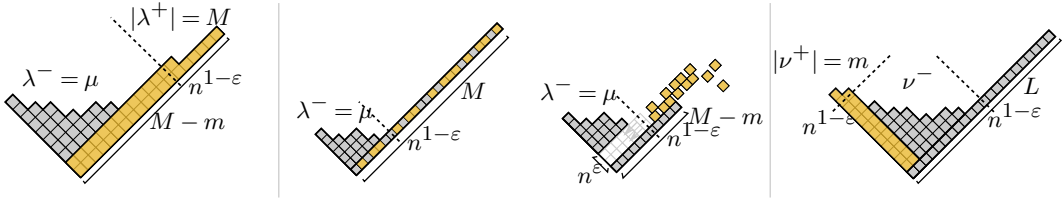

\begin{proof}
Given $\lnn$, set $\Rl:=\frac{|\lp|\log n}{\ell}$. In this proof, all little-o's are independent of $\Rl$. For all $\lnn$, by the last inequality in Lemma~\ref{lemma:factss}(\ref{f1}) and the inequality $\frac 1 {k!}\leq \left(\frac e k\right)^k$, we have
\begin{align}
\frac{1}{n!}f_\lambda^2
\leq\exp[|\lp|\log n-2\sum_{i=1}^p\lambda_i\log(\lambda_i)+2|\lp|].
\end{align}
Using the fact that for $i\in [1,p]$, $\lambda_i\geq n^{1-\eps}$, this bound gives
\begin{align}\label{eq_A}
\frac{1}{n!}f_\lambda^2
&\leq\exp[|\lp|\log n-2\sum_{i=1}^p(1-\eps)\lambda_i\log(n)+2|\lp|]\notag\\
&= \exp[-(1-2\eps)\Rl \ell+2\Rl \ell/\log n ].
\end{align}
On the other hand, by Lemma~\ref{lemma:factss}(\ref{f4}), if $C_\lambda\geq 0$, then 
\begin{align}\label{eq_B}
C_\lambda^{\ell}
&\leq \exp\bigg[\ell\log\bigg(\frac{|\lp|^2+n^{2-\eps}}2\bigg)\bigg]\notag\\
&= \exp\bigg[2\ell(\log \ell-\log\log n)+\ell\bigg(\log\bigg(\Rl^2+\frac{n^{2-\eps}\log^2n}{\ell^2}\bigg)-\log2\bigg) \bigg].
\end{align}

Combining~\eqref{eq_A} and~\eqref{eq_B}, and using~\eqref{eq_lowerbound_Zn_rough}, we obtain
\begin{align}
	\frac{\Z_n(\{\lambda\})}{H_{n,\ell}}
	&\leq \exp\bigg[\ell\bigg(2-(1-2\eps)\Rl+\frac{2\Rl }{ \log n}+\log\bigg(\Rl^2+o(1)\bigg)\bigg)+o(n)\bigg].
\end{align}
Now, the function $r\mapsto 2-(1-2\eps)r +\log(r^2)$ has a unique maximum on $\mathbb{R}_{>0}$ and goes to $-\infty$ on both ends, so it is less than $-1/100$ outside of a closed interval $I$, and in fact one can take $I=[\mt,\Mt]$. 
Hence for $n$ large enough and $\lnn$ with $\Rl\not\in [\mt,\Mt]$, $\P_{n,\ell}^+(\lambda)\leq \exp(-\ell/100+o(n))$, which entails the result since there are $e^{O(\sqrt n)}$ partitions of $n$. 
\end{proof}

\begin{lemma}[Uniqueness of the big part]\label{lemma:uniqueness}
Let $\lh$ be a random partition of $n$ under the assumptions of Theorem~\ref{thm:limitshape}. Then $\lh^+ = (\lh_1)$, and equivalently $\lh_2 \leq n^{1-\eps}$, w.h.p.
\end{lemma}
The proof of Lemma~\ref{lemma:uniqueness} requires a comparison of the contribution of partitions with a single ``big part'' with the contribution of those with more than one (indeed, because we have neither exact formulas nor precise estimates on the normalisations, we can only rely on ``comparison'' of probabilities at this stage). We will perform this comparison among partitions having the same ``small parts'' (called $\mu$), using the sets $\Lambda(\mu,M,m)$ defined at~\eqref{eq:Lambdadef} with $|\lambda^+|=M$ and $\lambda_1 = M-m$.
We will need the following two claims, whose proof is postponed to after that of the lemma.
\begin{claim}\label{fact_bound_g_content}
For all $\lambda\in\Lambda(\mu,M,m)$, we have
$C_\lambda\leq C_{\lambda^0} -(m-1)\frac M 2$.
\end{claim}
\begin{claim}\label{fact_bound_g_f}
If   $m>0$ then,
$\sum_{\lambda\in\Lambda(\mu,M,m)}f_\lambda\leq f_{\lambda^0}\exp[m(2\eps\log n +1)]$.
\end{claim}

\begin{proof}[Proof of Lemma~\ref{lemma:uniqueness}]
By Lemma~\ref{lem_control_big}, we know that, w.h.p., 
$|\lh^+|\in [\mt \frac{\ell}{\log n},\Mt \frac{\ell}{\log n}]$. We can thus assume this event for the rest of this proof.

We now condition on $|\lh^+|=M$ and $\lh^-=\mu$, with given $M\in  [ \mt \frac{\ell}{\log n},\Mt \frac{\ell}{\log n}]$ and $\mu\vdash n-M$. In the rest of the proof, the little-o's are uniform in $M$ and $\mu$ satisfying these conditions.
Let $m>0$. Combining Claims~\ref{fact_bound_g_content} and~\ref{fact_bound_g_f}, one obtains

\begin{align}{Z_n(\Lambda(\mu,M,m))}&=\sum_{\lambda\in\Lambda(\mu,M,m)}\frac{1}{n!}f_\lambda^2\Cl^\ell\leq\sum_{\lambda\in\Lambda(\mu,M,m)}\frac{1}{n!}f_\lambda^2\left(C_{\lambda^0} -(m-1)\frac M 2\right)^\ell\notag\\
&\leq \frac{1}{n!}\left(\sum_{\lambda\in\Lambda(\mu,M,m)}f_\lambda\right)^2\left(C_{\lambda^0} -(m-1)\frac M 2\right)^\ell\notag\\
&\leq \frac{1}{n!}\left(f_{\lambda^0}\exp[m(2\eps\log n +1)]\right)^2\left(C_{\lambda^0} -(m-1)\frac M 2\right)^\ell.
\end{align}
Hence
\begin{equation}\label{v1}\frac{Z_n(\Lambda(\mu,M,m))}{Z_n(\{\lambda^0\})}\leq \exp\left[\ell\log\left(1-\frac{(m-1)M}{2C_{\lambda^0}}\right)+2m(2\eps\log n +1)\right].
\end{equation}
But we know by Lemma~\ref{lemma:factss}(\ref{f4}) that
$C_{\lambda^0}\leq \frac{M^2}{2}+\frac{n^{2-\eps}}{2} = (1+o(1))\frac{M^2}{2}$
and
$M\leq \Mt \frac{\ell}{\log n}$.
Hence, applying these two inequalities successively to~\eqref{v1}, one gets
\begin{align}
\frac{Z_n(\Lambda(\mu,M,m))}{Z_n(\{\lambda^0\})}&\leq \exp\left[(1+o(1))\ell\log\bigg(1-\frac{m-1}{M}\bigg)+2m(2\eps\log n +1)\right]\notag\\
&\leq\exp\left[(1+o(1))\ell\log\bigg(1-\frac{(m-1)\log n}{\Mt \ell}\bigg)+2m(2\eps\log n +1)\right]. 
\end{align}
 Applying the inequality $\log(1-x)\leq x$ to the last bound gives
\begin{align} 
\frac{Z_n(\Lambda(\mu,M,m))}{Z_n(\{\lambda^0\})} &\leq\exp\left[-(1+o(1))\frac{(m-1)\log n}{\Mt }+2m(2\eps\log n +1)\right]\notag\\
&= \exp\left[(1+o(1))m  (4\eps-\frac 1 \Mt) \log n\right]. \end{align}
where in the final rewriting of the bound we use the fact that, since $m>0$, we have $m\geq n^{1-\eps}$ to include the terms $\log n$ and $2m$ in the $o(1)$.  
Now, recalling that  $\eps=\frac 1 {100}$, for $n$ large enough the last bound gives
\begin{align}\label{eq_exp_suppressed} \frac{Z_n(\Lambda(\mu,M,m))}{Z_n(\{\lambda^0\})} \leq \exp\left(-\frac{m\log n}{100} \right)\leq \exp\left(-n^{1-\eps}\right). 
\end{align}
Summing this over all $m>0$, we have
\begin{equation}
\sum_{m>0}Z_n(\Lambda(\mu,M,m))=o(Z_n(\{\lambda^0\}))
\end{equation}
which is enough to conclude that $\lh^+=(\lh_1)$ w.h.p.
\end{proof}

\begin{proof}[Proof of Claim~\ref{fact_bound_g_content}]
By Lemma~\ref{lemma:factss}~(iii), for any $\lambda \in \Lambda(\mu,M)$ we have a simple upper bound
\begin{equation} \label{eq:easycontsumbound}
 C_\lambda \leq C_{\lambda^+} -(n-M)+ C_{\mu}, 
 \end{equation} 
since $\lambda^+$ contains at least one part. For $\lambda^0$ the same lemma gives the equality 
\begin{equation} \label{eq:clambda0}
C_{\lambda^0} = \frac{M}{2}(M-1) - (n-M) + C_\mu.
\end{equation}
When $\lambda \in \Lambda(\mu,M, m)$, by Lemma~\ref{lemma:factss} (ii) we have $C_{\lambda^+}\leq M(M-m)/2$, and inserting this into~\eqref{eq:easycontsumbound} we have
 \begin{equation} 
 C_\lambda \leq \frac{M}{2}(M-m) -(n-M) + C_{\mu} = C_{\lambda^0} - \frac{M}{2}(m-1)
 \end{equation}
 as required.
\end{proof}

\begin{proof}[Proof of Claim~\ref{fact_bound_g_f}]
	To prove this claim we define a ``redistribution'' operation that enables us to compare the contribution of partitions with one big part to others.
Let $T$ be an SYT of shape $\lambda^0$, and consider the following operation.
\begin{longlist}
\item Create $n^\eps$ empty rows between the first row of $T$ and the rest,
\item choose $m$ numbers in the first row of $T$ \textit{($\binom M m$ choices)},
\item for each of these numbers, choose one of the newly created rows, and move it there \textit{($n^\eps$ choices each time)},
\item sort each row and delete the empty rows.
\end{longlist}
The output is either a SYT of some $\lambda\in\Lambda(\mu,M,m)$, or something else in the cases where the rows are not ordered by size or the labels do not increase along columns.
It is easily checked that this procedure can output any SYT of shape $\lambda$, for any $\lambda\in\Lambda(\mu,M,m)$ (indeed, for any $\lambda$, $\lp$ must have at most $\frac{n}{n^{1-\eps}}=n^\eps$ parts). Hence we have
\begin{align}
\sum_{\lambda\in\Lambda(\mu,M,m)}f_\lambda&\leq \binom{M} m n^{\eps m}f_{\lambda^0}
	\leq \frac{M^m}{m!} n^{\eps m}f_{\lambda^0}
\leq f_{\lambda^0}\exp[m(2\eps\log n +1)]
\end{align}
where in the last inequality we used the bound $m!\geq (m/e)^m$, along with the facts that $\log M \leq \log n$ and $\log m\geq (1-\eps)\log n$.
\end{proof}

\subsection{Asymptotic estimate for  \texorpdfstring{$H_{n,\ell}$}{} and macroscopic behaviour of  \texorpdfstring{$\lh$}{}}

Lemmas~\ref{lem_control_big} and~\ref{lemma:uniqueness} imply that $\lh\in\Lambda^*$ w.h.p. (using the notation introduced at~\eqref{eq:Lambdastardef}) and hence
\begin{equation}\label{aeu}
H_{n,\ell} =(1+o(1))Z_n(\Lambda^*),
\end{equation}
and it follows from the proofs that the little-o is uniform for $\ell/n$ in any compact subset of $(0,\infty)$
From this result, we obtain  Theorem~\ref{thm:asymptotic_H_nl} and parts (i) and (ii) of Theorem~\ref{thm:limitshape}.

\begin{proof}[Proof of Theorem~\ref{thm:asymptotic_H_nl}]

Take $\lambda\in \Lambda^*$. We have, by Lemma~\ref{lemma:factss}(\ref{f4}) 
\begin{equation}\label{uuu0}
  C_\lambda\leq \frac{\lone^2}{2}+\frac{n^{2-\eps}}{2}=(1+o(1))\frac{\lone^2}{2}  
\end{equation} where the $o(1)$ is uniform over all $\lambda\in \Lambda^*$, and more generally, from now on, all little-o's and big-O's will be uniform over all $\lambda\in \Lambda^*$ when applicable.

On the other hand, by Lemma~\ref{lemma:factss}(\ref{f1}) we have $\frac 1 {n!}f_\lambda^2\leq\frac{n^{\lambda_1}}{(\lone!)^2}$. But the Stirling approximation plus the fact that $\lone=O\left(\frac n {\log n}\right)$ yields $\lone=\exp(\lone\log n +o(n))$
hence
\begin{equation}\label{uuu1}
    \frac 1 {n!}f_\lambda^2\leq\exp(-\lone\log n +o(n)).
\end{equation}
Combining~\eqref{uuu0} and~\eqref{uuu1} yields
\begin{align}\label{uuu}
Z_n(\{\lambda\})=\frac 1 {n!}f_\lambda^2C_\lambda^\ell \leq\exp\left[2\ell\log(\lone)-\ell\log 2-\lone\log n +o(n)\right].
\end{align}
Substituting $\lone=\frac{\Rl \ell}{\log n}\sim\frac{\Rl 2 \theta n }{\log n}$ into the inequality above, we obtain
	\begin{align} \label{eq:upperbdz}
Z_n(\{\lambda\})
	&\leq \bigg( \frac{\ell}{\log n}\bigg)^{2 \ell} \exp\left[(2\log\Rl-\log2-\Rl)\ell+o(n)\right].
\end{align}
Now, since the function on the positive reals $x \mapsto 2 \log x - x$ has a unique maximum at $x = 2$, we have 
\begin{align}
Z_n(\{\lambda\})
&\leq \bigg( \frac{\ell}{\log n}\bigg)^{2 \ell} \exp \big[(-2 + \log 2) \ell+o(n) \big],
\end{align}
and since there are  $e^{O(\sqrt n)}$ partitions of $n$, $Z_n(\Lambda^*) \leq  \max_{\lambda \in \Lambda^*} Z_n(\{\lambda\})e^{O(\sqrt n)} $. 
Together with the lower bound of Lemma~\ref{lem_lowerbound_Z_n} and~\eqref{aeu} this proves that 
\begin{equation}
	H_{n,\ell} = \bigg( \frac{\ell}{\log n}\bigg)^{2 \ell} \exp \big[ ( -2+ \log 2) \ell+o(n) \big],
\end{equation}
as required.
\end{proof}

\begin{proof}[Proof of Theorem~\ref{thm:limitshape}, part (i)]
The upper bound~\eqref{eq:upperbdz} in the previous proof along with Lemma~\ref{lem_lowerbound_Z_n} implies that for $\lambda\in\Lambda^*$ (uniformly),
\begin{equation}\P^+_{n,\ell}(\lambda)\leq \exp\left[\ell(2\log\Rl-2 \log2+2-\Rl)+o(n)\right].
\end{equation}
Any non-negligible deviation of $R_{\lambda}$ from the unique maximiser of this upper bound thus entails an exponentially decreasing probability, which is enough to conclude that, under the Plancherel--Hurwitz measure at high genus, $R_{\lh}\xrightarrow{p} 2$, which is what we wanted.
\end{proof}

\begin{proof}[Proof of Theorem~\ref{thm:limitshape}, part (ii)]
We refine the previous upper bound on $\P^+_{n,\ell}(\lambda)$ for $\lambda = \lambda_1 \sqcup \tilde{\lambda} \in \Lambda^*$ by writing $\frac{1}{n!}f_\lambda^2\leq\frac{n!}{\lambda_1!^2(n-\lambda_1)!^2}f_{\tilde \lambda}$ by Lemma~\ref{lemma:factss} part~(\ref{f1}), along with $C_{\lambda}=(1+o(1))\frac{\lone^2}{2}$ as above, uniformly for $\lambda\in \Lambda^*$. 

Recall that $\Rl=\frac{\lone \log n}{\ell}$. By Lemma~\ref{lem_lowerbound_Z_n} and the Plancherel entropy estimate~\eqref{eq:plancherel_integral} for $\frac{1}{(n-\lambda_1)!}f_{\tilde{\lambda}}^2$, we have
 \begin{align}\P^+_{n,\ell}(\lambda)&\leq \exp\left[\ell(2\log\Rl-2 \log2+2-\Rl)-n(1 + 2I_{\text{hook}}(\psi_{\tilde{\lambda}}))+o(n)\right]\notag\\&\leq \exp\left[-n(1+2I_{\text{hook}}(\psi_{\tilde{\lambda}}))+o(n)\right]
 \end{align}
for all $\lambda\in\Lambda^*$, uniformly (where the second inequality comes from the unique maximum at $R_\lambda = 2$). We recognise the Plancherel measure estimate~\eqref{eq:plancherel_integral}. 
	Since $\lh\in\Lambda^*$ w.h.p., this implies, as in the classical Plancherel case (see \cite[Section 1.17]{Romik_2015}), the almost sure convergence in supremum norm to the VKLS limit shape~\eqref{eq:plancherel_limit_profile}.
\end{proof}
	
	Note that since the function $x\mapsto \Omega(x)-|x|$ has support $[-2,2]$, the convergence of the profile directly implies the ``lower bound'' in Theorem~\ref{thm:limitshape}~(iii): 
\begin{equation}\label{eq_lowr_bound_second_part}
\min(\lh_2,\ell(\lh))\geq (2-o_p(1))\sqrt n. \qedhere
\end{equation}
The upper bound is more delicate and is the subject of the next section.

\subsection{First bound on the number of parts}

To conclude this section, we use Theorem~\ref{thm:limitshape}~(i) to prove a further macroscopic feature of the limit behaviour, which completes the rough bounding box one side of which is determined by Lemma~\ref{lemma:uniqueness}. 
\begin{lemma}[Bounding the length above]\label{lemma:lengthbd}
Let $\lh$ be a random partition of $n$ under the assumptions of Theorem~\ref{thm:limitshape}. Then $\ell(\lh)\leq n^{1-\eps}$ w.h.p.
\end{lemma}
\begin{proof}
Take $\lambda\in \Lambda^*$,  we first use arguments similar to Lemma~\ref{lem_control_big} to control the number of boxes in the big parts of the conjugate partition $\lambda^\prime$. Let us write $\lambda = L \sqcup \nu^\prime$ (as in the proof of Lemma~\ref{lemma:uniqueness}, the little-o's and big-O's are uniform in $L$ and $\nu$) so the conjugate of the small parts $\lambda^-$ is
\begin{equation}
    \nu = (\lambda_2, \lambda_3, \ldots, \lambda_{\ell(\lambda)})^\prime  = \nu^+ \sqcup \nu^-
\end{equation}
where $\nu^+$ and  $\nu^-$ denote respectively the parts of $\nu$ that are greater and less than $n^{1-\eps}$, and $\ell(\nu) \leq n^{1-\eps}$ (see Figure~\ref{fig:split_partition}). Suppose that the big parts have size $m =m(\lambda)= |\nu^+|$ (note that if $m>0$ then $m\geq n^{1-\eps}$). By Lemma~\ref{lemma:factss}~(i) and the Stirling approximation, we have
\begin{equation}
    \frac{1}{n!}{f_\lambda}^2 \leq \frac{n!}{(n-L)!^2L!^2} {f_{\nu}}^2 \leq 
    \frac{1}{(n-L)!} {f_{\nu}}^2 n^{-L}\exp\left[ O\left( \frac{n}{\log n}\right)\right],
\end{equation}
but by~\eqref{eq_A}
\begin{equation}
 \frac{1}{(n-L)!}f_{\nu}^2 \leq \exp[-(1-2\eps)m \log (n-L) + 2m].
\end{equation}
Since $C_\lambda\leq \frac{1}{2}(L^2+n^{2-\eps})$ (by Lemma~\ref{lemma:factss}(\ref{f4})), we get
\begin{equation}
Z_n(\{\lambda\})\leq n^{-L}\left(\frac{L^2}{2}\right)^\ell\exp\left[-(1-2\eps)m \log (n) + O(m)+O\left( \frac{n}{\log n}\right)\right].
\end{equation}
Using~\eqref{eq_lowerbound_Zn_rough} to bound $H_{n,\ell}$ below we have
\begin{align}
\frac{Z_n(\{\lambda\})}{H_{n,\ell}} \leq \exp\left[-\frac{98 m}{100}\log n + O(m) +  O\left( \frac{n}{\log n}\right) \right].
\end{align} 
Thus, since there are $e^{O(\sqrt n)}$ partitions of $n$,
\begin{equation}\label{rr1}
    Z\left( \left\{\lambda \in \Lambda^* \middle\vert m(\lambda)\geq \frac{n}{\sqrt{\log n}} \right\}\right)=o(H_{n,\ell}).
\end{equation}

Now, take $\lambda\in \Lambda^*$ with $m(\lambda)<\frac{n}{\sqrt{\log n}}$. Write $\lambda = L\sqcup \nu^\prime$ as above, fix $\nu^- := \mu^\prime$, $|\nu^+| = m$ and fix $M = L + m$. In analogy with the proof of Lemma~\ref{lemma:uniqueness}, for given $M$ and $m$, we consider the set $\hat{\Lambda}(\mu,M,m)$ of all such partitions, and compare  $Z_n(\hat{\Lambda}(\mu,M,m))$ to  $Z_n(\{\lambda^0\})$, where $\lambda^0 = M \sqcup \mu$ as before, such that $\hat{\Lambda}(\mu,M,0) = \{\lambda^0\}$. It follows immediately from this definition that Claim~\ref{fact_bound_g_f} also applies to $\hat{\Lambda}(\mu,M,m)$, that is to say 
\begin{equation}
\sum_{\lambda\in\hat{\Lambda}(\mu,M,m)}f_\lambda\leq f_{\lambda^0}\exp[m(2\eps\log n +1)].
\end{equation}
Then, by Lemma~\ref{lemma:factss}~(iii) twice and~\eqref{eq:clambda0}, we get
\begin{align}\label{cccb0}
    \Cl&=
    \frac{L(L-1)}{2}-(n-L)-C_\nu\notag\\
    &= \frac{L(L-1)}{2}-(n-L)-C_{\nu^+}+ \ell(\nu^+)(n-M)+C_\mu\notag\\
   &= C_{\lambda^0} - \frac{m(2M-m+1)}{2} + \ell(\nu^+)(n-M)-C_{\nu^+}.
\end{align}
We have by Lemma~\ref{lemma:factss}(\ref{f3})
\begin{equation}\label{cccb}
    C_{\nu^+}=\sum_{i=1}^{\ell(\nu^+)} \frac{\nu_i(\nu_{i} - 2i +1)}{2}   \geq \sum_{i=1}^{\ell(\nu^+)} \frac{\nu_i(n^{1-\eps} - 2\ell(\nu^+) +1)}{2}\geq 0  
\end{equation}
for $n$ large enough.
Now, it is clear that $\ell(\nu^+) \leq m/n^{1-\eps}$ hence $\ell(\nu^+)(n-M)=o(mM)$.  Since $m = o(M)$, combining these facts with~\eqref{cccb0} in~\eqref{cccb}, we obtain 
\begin{align}
    \Cl\leq C_{\lambda^0} - mM(1+o(1)).
\end{align}
As $C_{\lambda^0} \leq (1+o(1)) \frac{M^2}{2}$, repeating precisely the arguments of Lemma~\ref{lemma:uniqueness} we have 
\begin{align}
\frac{Z_n(\hat{\Lambda}(\mu,M,m))}{Z_n(\{\lambda^0\})} &\leq \exp \bigg[\ell \log \left(1 - \frac{m \log n }{3 \ell} \right)(1+o(1)) + 2m(2\eps \log n +1) \bigg] \notag\\&\leq \exp \left(-\frac{m\log n}{4}(1 + o(1))\right)\leq \left(-n^{1-\eps}\frac{\log n}{4}(1 + o(1))\right).
\end{align}

Since there are $e^{O(\sqrt{n})}$ partitions of $n$ and $Z_n(\{\lambda^0\}\leq H_{n,\ell}$ we obtain
\begin{equation}\label{rr2}
    Z\left( \left\{\lambda \in \Lambda^* \middle\vert m(\lambda)< \frac{n}{\sqrt{\log n}} \right\}\right)=o(H_{n,\ell}).
\end{equation}

Equations~\eqref{rr1} and ~\eqref{rr2} immediately imply that the probability that $m(\lh)>0$ is $o(1)$, so  $\ell(\lh) \leq n^{1-\eps}$ w.h.p. as required.
\end{proof}

\section{Proofs: Microscopic features}\label{sec:microproofs}

In this section we consider the smaller scale of the limit shape and prove Theorem~\ref{thm:limitshape} part~(\ref{thm:lst_2_sqrtn}) by bounding the size of $\lh_2$ and $\ell(\lh)$ above. As previously mentioned, the VKLS limit shape result in supremum norm does not imply such a bound, and even in the Plancherel case extra arguments are needed to obtain the sharp bound $(2+o_p(1))\sqrt{n}$ on $\lh_1$. 
In the Plancherel case, a good way to do this is to use the corner-growth process mentioned in the introduction, which provides one with a coupling between the measures at sizes $n$ and $n-1$, through which the evolution of $\lambda_1$ is tractable inductively (see e.g. \cite[Section 1.19]{Romik_2015}).

In our context, we do not have such a coupling, however we will be able to compare the behaviour of random partitions in sizes $n$ and $n-1$ by calculations which in some sense provide an approximation of the corner-growth process.

\medskip 
From now on we will work with a given value of $\lone=L$ with $L \in \intvl$, and in the notation of Theorem~\ref{thm:limitshape}, we let $\tl=\lambda\setminus \lone$. We introduce notation for the normalisation with a fixed first part, putting 
\begin{equation}Z_{n}[L] := Z_n(\{\lambda \vdash n| \lambda_1 = L \}) . \label{eq:znLnotation} \end{equation}

\subsection{An intermediate bounding box for \texorpdfstring{$\tilde{\lh}$}{ltilde} }

\begin{prop}\label{prop_bounding_box}
	Under the hypotheses of Theorem~\ref{thm:limitshape}, 
 the second part satisfies $\lh_2 \leq (e+o_p(1)) \sqrt{n}$, and similarly the length satisfies $\ell(\lh)\leq (e+o_p(1)) \sqrt{n}$.
\end{prop}
 Our proof relies on comparisons between partitions of $n$ and partitions of $n-1$. 
\begin{lemma} \label{lem:compare_normalizations}
Uniformly for
$L \in \intvl$,  we have
\begin{equation} \label{eq:norm_ratio}
1- O(1/\log n)\leq \frac{Z_{n-1}[ L]}{Z_n[L]}\leq 1 + O(\log^2 n/n^\eps).
\end{equation}
\end{lemma}

\begin{proof}
We first bound $Z_{n-1}[ L]/Z_n[L]$ above. For partitions of $n-1$ whose first part is equal to $L$, the normalisation of the conditioned measure is 
\begin{equation}
Z_{n-1}[ L]  = \frac{1}{(n-1)!}\sum_{\substack{\mu \vdash n-1\\\mu_1 = L}} {f_{\mu}}^2 {C_\mu}^\ell =  \frac{1}{n!}\sum_{\substack{\mu \vdash n-1\\\mu_1 = L}} f_{\mu} {C_\mu}^\ell \sum_{\lambda: \mu \nearrow \lambda} f_\lambda,
\end{equation}
where the final equality comes from the normalisation of the Plancherel growth process, as given in Proposition~\ref{prop:pgp_norm}.
	This can equally be expressed as a sum over partitions of $n$ with first part equal to $L$; splitting the sum according to whether the additional box is on the first part or not, we have
\begin{align}
	Z_{n-1}[ L]  = \frac{1}{n!} \sum_{\substack{\lambda \vdash n\\ \lambda_1 = L}} f_{\lambda}  \sum_{\mu: \mu \nearrow \lambda, \mu_1 = L } f_\mu {C_\mu}^\ell + \frac{1}{n!}\sum_{\substack{\mu \vdash n-1\\\mu_1 = L}} f_{\mu^{1+}}f_{\mu} {C_\mu}^\ell 
\end{align}
where $\mu^{1+}$ denotes the partition $\mu^{1+} = (\mu_1 + 1, \mu_2, \mu_3 , \ldots)$. 

We start by considering the first term. By~\eqref{eq_exp_suppressed} we know that the contribution to $Z_{n-1}[L]$ from any partition with more than one big part will be very small, i.e.
\begin{align}\label{x0}
	Z_{n-1}[ L]\left(1+O\left(\exp(-n^{1-\eps}\right)\right)  &= \frac{1}{n!} \sum_{\substack{\lambda \vdash n\\ \lambda_1 = L\\\lambda_2<n^{1-\eps}}} f_{\lambda}  \sum_{\mu: \mu \nearrow \lambda, \mu_1 = L } f_\mu {C_\mu}^\ell \\ & + \frac{1}{n!}\sum_{\substack{\mu \vdash n-1\\\mu_1 = L}} f_{\mu^{1+}}f_{\mu} {C_\mu}^\ell . \notag
\end{align}
 Then, for $\mu \nearrow \lambda$, $\mu_1= \lambda_1 = L$ and $\mu_2 < n^{1-\eps}$, 
	the modulus of the content of the additional box will be at most $n^{1-\eps}$, and 
\begin{equation}\label{x1}
{C_\mu}^\ell = (C_\lambda + O(n^{1-\eps}))^\ell= {C_\lambda}^\ell\left(1 + O\left(\frac{n^{1-\eps}}{C_\lambda}\right)\right)^\ell .
\end{equation}
But for such a $\lambda$, by Lemma~\ref{lemma:factss}(\ref{f3}) then (\ref{f2})
\begin{equation}\label{x2}
   C_\lambda\geq \frac{L(L-1)}{2}-n-|C_{\lm}|\geq (1+o(1))\frac{L^2}{2},
\end{equation}
hence combining~\eqref{x1} and~\eqref{x2} we obtain
\begin{equation}\label{x3}
{C_\mu}^\ell = {C_\lambda}^\ell (1+O(\log^2 n /n^{\eps}) ).
\end{equation}
Then, applying this to~\eqref{x0} we have
\begin{align}
Z_{n-1}[ L]\left(1+O\left(\exp(-n^{1-\eps}\right)\right)  \leq  &\frac{1}{n!} \sum_{\substack{\lambda \vdash n\\ \lambda_1 = L}} f_{\lambda}  C_\lambda^\ell (1+O(\log^2 n /n^{\eps}) )\sum_{\mu: \mu \nearrow \lambda} f_\mu \\
&+ \frac{1}{n!}\sum_{\substack{\mu \vdash n-1\\ \mu_1 = L}} f_{\mu^{1+}}f_{\mu} {C_\mu}^\ell \notag
\end{align}
where we over-count by adding the partition $(L-1,\lambda_2, \lambda_3, \ldots)$  in the sum over $\{\mu : \mu \nearrow \lambda \}$ and re-adding partitions $\lambda$ with $\lambda_2\geq n^{-1\eps}$ to the first sum after applying~\eqref{x3}. Then, using the identity $\sum_{\mu: \mu \nearrow \lambda} f_\mu = f_\lambda$ from recursively counting SYT, we have
\begin{align} \label{eq:znl-intermediate-ratio}
Z_{n-1}[ L] \left(1+O\left(\exp(-n^{1-\eps}\right)\right) \leq  Z_{n}[ L] (1+O(\log^2 n/ n^{\eps}) ) 
+ \frac{1}{n!}\sum_{\substack{\mu \vdash n-1\\ \mu_1 = L}} f_{\mu^{1+}}f_{\mu} {C_\mu}^\ell. 
\end{align}
The second term on the right is finally estimated using the hook-length formula~\eqref{eq:hooklength}. For $\mu \vdash n-1$, $\mu_1 = L$ and $\mu_2 \leq n^{1-\eps}$ we have 
	\begin{equation} \label{eq:boundfoneplus}
f_{\mu^{1+}} = n f_\mu  \exp \bigg[ \sum_{j = 1}^{\mu_2} \log \frac{L + \mu_j^\prime - j}{L+ \mu_j^\prime-j +1}   +  \sum_{j = \mu_2+1}^L \log \frac{L-j+1}{L-j+2} \bigg] = n f_\mu O (\log n / n)
\end{equation}
so the second term is just $Z_{n-1}[ L] O (\log n / n)$, which is absorbed into the left hand side of~\eqref{eq:znl-intermediate-ratio}, proving the upper bound in~\eqref{eq:norm_ratio}. 

The ratio is similarly bounded below by writing $Z_n[L]$ as a sum over partitions of $n-1$, to find 
\begin{align}
Z_{n}[ L]\left(1+O\left(\exp(-n^{1-\eps}\right)\right)  \leq &  \frac{1}{(n-1)!}\sum_{\substack{\mu \vdash n-1\\ \mu_1 = L}} {f_{\mu}}^2 {C_\mu}^\ell (1- O(\log^2n/n^\eps)) \\
&+  \frac{1}{n!}\sum_{\substack{\lambda \vdash n\\ \lambda_1 = L}} f_{\lambda^{1-}}f_{\lambda } {C_{\lambda }}^\ell \notag
\end{align}
where $\lambda^{1-}$  denotes the partition $(\lambda_1-1,\lambda_2, \lambda_3,\ldots)$.  Then we have 
\begin{align} \label{eq:boundfoneminus}
f_{\lambda ^{1-}} &= \frac{1}{n} f_\lambda   \exp \bigg[ \sum_{j = 1}^{\lambda_2} \log \frac{L + \lambda_j^\prime - j}{L+ \lambda_j^\prime-j -1}  +  \sum_{j = \lambda_2+1}^{L-1} \log \frac{L-j+1}{L-j} \bigg]  \\
&= \frac{1}{n} f_\lambda O (n / \log n) \notag
\end{align}
and finally $Z_{n-1}[ L]/Z_{n}[ L] \geq 1 - O(1/\log n)$, which concludes the proof.
\end{proof}

\begin{figure}
\centering
{\def\svgwidth{\sizeofpage}
\input{for-remove-2nd-part.tex} 
}
\caption{Left, partitions $L\sqcup k \sqcup \mu\vdash n$ and $L\sqcup  \mu \vdash n-k$, as used in the proof of Proposition~\ref{prop_bounding_box} to bound the second part $\lambda_2$. Right, partitions $L\sqcup (k \sqcup \mu)^\prime \vdash n$ and $L\sqcup  \mu^\prime \vdash n-k$ with a box removed from each as used to bound the length $\ll$.} \label{fig:remove-2nd-part}
\end{figure}
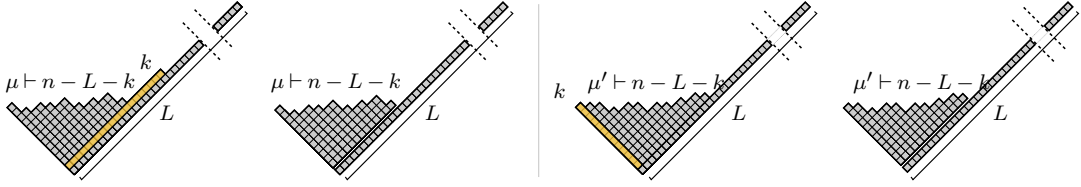

We are now ready to prove Proposition~\ref{prop_bounding_box}.
\begin{proof}[Proof of Proposition~\ref{prop_bounding_box}]
Under the Plancherel--Hurwitz measure  conditioned on the first part being $\lh_1 = L$ (with $L \in\intvl$), the distribution of the second part is
\begin{equation}
\P_{n,\ell}^+(\lh_2 = k| \lh_1 = L) = \frac{1}{n! Z_n[L]} \sum_{\substack{\mu \vdash n - L - k\\ \mu_1 \leq k}} f^2_{L \sqcup k \sqcup \mu} C^\ell_{L \sqcup k \sqcup \mu}. 
\end{equation}
Comparing SYT of shape $L \sqcup k \sqcup \mu \vdash n$ with ones of shape $L \sqcup \mu \vdash n -k$, obtained by removing the second part, we have a rough upper bound of
\begin{equation}
f_{L \sqcup k \sqcup \mu } \leq \binom{n}{k} f_{L \sqcup \mu}
\end{equation}
by over-counting the ways the boxes of the second part could be labelled. The content-sum of each of these partitions are related by
\begin{equation}
C_{L \sqcup k \sqcup \mu } = C_{L\sqcup \mu} - |\mu| + \frac{k(k-3)}{2} = C_{L\sqcup \mu}(1+o(1))
\end{equation}
as long as $k\leq n^{1-\eps}$.
By Lemma~\ref{lemma:uniqueness} we may consider only partitions with $\lambda_2\leq n^{1-\eps}$, we thus get
	\begin{align}
\P_{n,\ell}^+(\lh_2 = k| \lh_1 = L) \leq &\frac{1}{n! Z_n[L]}\sum_{\mu \vdash n - L - k} \binom{n}{k}^2 f^2_{L \sqcup \mu} C^\ell_{L \sqcup \mu} (1+o(1)) \notag\\
&= \binom{n}{k}^2 \frac{(n-k)! }{n! }\frac{Z_{n-k}[ L]}{Z_n[L]}(1+o(1)).
\end{align}
Now, Lemma~\ref{lem:compare_normalizations} implies $\frac{Z_{n-k}[ L]}{Z_n[L]}=(1+o(1))^k=e^{o(k)}$ , hence (by the inequality $\frac{1}{k!}\leq (e/k)^{k} $)
\begin{equation}
\P_{n,\ell}^+(\lh_2 = k| \lh_1 = L) \leq \frac{n!}{k!^2(n-k)!}e^{o(k)}\leq \frac{n^k}{(k/e)^{2k}} e^{o(k)},
\end{equation}
where the little-o is uniform for $L\in\intvl$. 
Hence
\begin{equation}
\forall \epsilon>0, \qquad 
\P_{n,\ell}^+(\lh_2 = (1+ \epsilon)e \sqrt{n}|\lh_1 = L) \leq  (1+\eps)^{-2e\sqrt{n} + o(\sqrt{n})},
\end{equation}
with the same uniformity.
From Theorem~\ref{thm:limitshape}, part (i), this implies that
 $\lambda_2 \leq e(1+o_p(1)) \sqrt{n}$ as $n\to\infty$.

The proof that $\ell(\lh) \leq e(1+o_p(1)) \sqrt{n}$ works exactly the same way, except that, instead of removing the second part, one removes one box from each part except the first one (see Figure~\ref{fig:remove-2nd-part}), which explicitly means comparing partitions $L \sqcup (k \sqcup \mu)^\prime$ and $L \sqcup \mu^\prime$ to find, by Lemma~\ref{lemma:lengthbd}, 
\begin{align}
\P^+_{n,\ell}(\ell(\lh) = k+1| \lh_1 = L) =& \frac{1}{n! Z_n[L]} \sum_{\substack{\mu \vdash n - L - k\\ \mu_1 \leq k}} f^2_{L \sqcup (k \sqcup \mu)^\prime} C^\ell_{L \sqcup (k \sqcup \mu)^\prime} \notag \\
\leq& \frac{1}{n! Z_n[L]}\sum_{\mu \vdash n - L - k} \binom{n}{k}^2 f^2_{L \sqcup \mu} C^\ell_{L \sqcup \mu} (1+o(1))
\end{align}
leading to the upper bound by precisely the same steps as for the second part.
\end{proof}

\subsection{Final bounding box for \texorpdfstring{$\tilde{\lh}$}{ltilde}}

Thanks to the previous results we can now focus (for $n$ large enough) on partitions $\lambda = L \sqcup \tilde{\lambda}$ such that $\lambda_2,\ell(\lambda)\leq 2e\sqrt n$, which by Lemma~\ref{lemma:factss}~(i) implies $C_\tl\leq en^{3/2}$.

\begin{prop}
We have 
\begin{equation}\label{eq_proba_rest}
\P_{n,\ell}^+((\lh_2,\lh_3,\ldots) = \tl|\lh_1 = L)\leq\exp\left[O\left(\frac{C_\tl\log^2 n }{n}+\log^2 n\right)\right]\frac{(f_\tl)^2}{(n-L)!}
\end{equation}
uniformly for $L\in\intvl$ and all $\tl\vdash n-L$ satisfying $C_\tl\leq en^{3/2}$.
\end{prop}

\begin{proof}
Take any $\lambda$ satisfying the conditions above. It is easily shown from the hook-length formula that
\begin{equation}\fl\leq \binom{n}{L}f_\tl 
\end{equation}
and also
\begin{equation}\Cl\leq \frac{L(L-1)}{2}+C_\tl = \frac{L(L-1)}{2}\left(1+\frac{2 C_\tl}{L(L-1)}\right).\end{equation}
Hence
\begin{equation}
\Zl\leq  \frac{1}{n!}\binom{n}{L}^2f_\tl^2\left(\frac{L(L-1)}{2}\right)^\ell \left(1+\frac{2 C_\tl}{L(L-1)}\right)^\ell,
\end{equation}
and using the bound $(1+x)^k \leq e^{x k}$ along with the estimates $\ell = O(n)$ and $1/L = O(\log n/n)$ we have
\begin{equation}
\Zl \leq \frac{(n-L)!}{n!}\binom{n}{L}^2\left(\frac{L(L-1)}{2}\right)^\ell \exp\left[O\left(\frac{C_\tl \log^2 n }{n}\right)\right]\frac{(f_\tl)^2}{(n-L)!}. \label{st1}
\end{equation}

Let us now bound the normalisation factor $Z_{n}[L]$ below. 
Let \begin{equation}
    \Lambda_L:=\{\mu\vdash n-L|C_\mu\geq 0\text{ and } \mu_1,\ell(\mu)\leq 3\sqrt n\}. 
    \end{equation}
Take $\lambda=L\sqcup \mu$ with $\mu\in \Lambda_L$. In the rest of the proof, uniformity will also be over $\mu\in \Lambda_L$.
Then, by Lemma~\ref{lemma:factss}-(iii),
\begin{equation}\label{o}
    C_\lambda\geq \frac{L(L-1)}{2}-(n-L)=\frac{L(L-1)}{2}\left(1+O\left(\frac{\log^2 n}{n}\right)\right)
\end{equation}
and, by the hook-length formula,
\begin{align}
 f_\lambda=\frac{n!}{(n-L)!}\frac{1}{(L-\mu_1)!}\frac{1}{\prod_{i=1}^{\mu_1}(\mu_i'+L-i)} f_\mu.
\end{align}
But 
\begin{align}
   {\prod_{i=1}^{\mu_1}(\mu_i'+L-i)}&\leq \prod_{i=1}^{\mu_1}(L-i +3\sqrt n) \notag\\
   &=\prod_{i=1}^{\mu_1}(L-i)(1 +\frac{3\sqrt n}{L-i}) \notag\\
   &\leq \prod_{i=1}^{\mu_1}(L-i)(1 +\frac{4 \sqrt n}{L})\quad \text{for $n$ large enough} \notag\\
   &= \exp(O(\log n))\prod_{i=1}^{\mu_1}(L-i).
\end{align}
Hence
\begin{equation}
    f_\lambda\geq \binom n L f_\mu \exp(O(\log n)).
\end{equation}

By Theorem~\ref{thm:vkls} (together with the fact that $\P_{n,0}(\Cl\geq 0)\geq 1/2)$), we have 
\begin{equation}
    \sum_{\mu\in \Lambda_L}f_\mu^2\geq \left(\frac 1 2 -o(1)\right)(n-L)!
\end{equation}
hence
\begin{equation}
   \sum_{\substack{\lambda=L\sqcup\mu\\\mu\in \Lambda_L}  }f_\lambda^2\geq{(n-L)!}\binom{n}{L}^2 \exp(O(\log n)),
\end{equation}
and therefore, using~\eqref{o},
\begin{equation}\label{st2}
    \Z_n[L]\geq \Z_n(\{\lambda=L\sqcup \mu|\mu\in\Lambda_L\})\geq \frac{(n-L)!}{n!}\binom{n}{L}^2\left(\frac{L(L-1)}{2}\right)^\ell \exp(O(\log^2 n)).
\end{equation}

Combining~\eqref{st1} and~\eqref{st2} proves~\eqref{eq_proba_rest}.
The uniformity condition is easily checked.\end{proof}

We can now prove the last part of Theorem~\ref{thm:limitshape}.
\begin{proof}[Proof of Theorem~\ref{thm:limitshape}, part~(\ref{thm:lst_2_sqrtn})]
Set $L \in \intvl$, $m=n-L$ and $\eps>0$.

Let us now introduce $\mathcal Z_m=\{\mu\vdash m|\mu_1,\ell(\mu)\leq 3\sqrt m\}$,  $\contentsetm{\leq t}=\{\mu\in\mathcal Z_m  |C_\mu\leq m^{3/2}t\} $ and respectively $\contentsetm{\geq t}=\{\mu\in \mathcal Z_m|C_\mu\geq m^{3/2}t\}$), and $\firstsetm{\eps}=\{\mu\in \mathcal Z_m |\max(\mu_1,\ell(\mu))\geq (2+\eps)\sqrt m\}.$ 

In accordance with \eqref{eq_proba_rest}, we will show that for any constant $C$, uniformly in $C$ and $L$, 
\begin{equation}\label{a}
\sum_{\mu\in \firstsetm{\eps}} \exp\left(C\left(\frac{ C_\mu \log^2 n}{n}+\log^2 n\right)\right)\frac{(f_\mu)^2}{m!}\leq \exp[-(1+o(1))K\sqrt n]
\end{equation}
where $K$ is the constant of Proposition~\ref{prop:DZ}.

Let us fix $t=\frac 1 {\log^3 n}$. To establish this inequality we will split the sum into two sums, over $\firstsetm{\eps}\cap \contentsetm{\leq t}$ and $\firstsetm{\eps}\cap \contentsetm{\geq t}$.

For all $\mu \in \firstsetm{\eps}\cap \contentsetm{\leq t}$, we have $\frac{C_\mu \log^2 n }{n}+\log^2 n=o(\sqrt n)$ uniformly in $\mu$ hence by Proposition~\ref{prop:DZ}, we have 

\begin{equation}\label{a1}
\sum_{\mu\in \firstsetm{\eps}\cap \contentsetm{\leq t}} \exp\left(C\left(\frac{ C_\mu \log^2 n}{n}+\log^2 n\right)\right)\frac{(f_\mu)^2}{m!}\leq \exp[-(1+o(1))K\sqrt n].
\end{equation}

On the other hand, for all $\mu \in \firstsetm{\eps}\cap \contentsetm{\geq t}$, we have both $C_\mu =O(n^{3/2})$ and, by Theorem~\ref{thm:deviationContents} (very roughly), $\frac{(f_\mu)^2}{m!}\leq \exp(-n^{0.9})$. Since also $|\firstsetm{\eps}\cap \contentsetm{\geq t}|=e^{O(\sqrt n)}$, we have

\begin{equation}\label{a2}
\sum_{\mu\in \firstsetm{\eps}\cap \contentsetm{\geq t}} \exp\left(C\left(\frac{ C_\mu \log^2 n}{n}+\log^2 n\right)\right)\frac{(f_\mu)^2}{m!}\leq \exp[-(1+o(1))n^{0.9}].
\end{equation} 

Putting \eqref{a1} and \eqref{a2} together establishes \eqref{a}.

Now, by Proposition~\ref{prop_bounding_box}, we have $\tilde{\lh}\in \mathcal{Z}_m$ w.h.p., hence we can combine~\eqref{a} with~\eqref{eq_proba_rest} to establish that for all $\eps>0$ there is $K>0$ such that
\begin{equation}\P_{n,\ell}^+(\max(\lh_2,\ell(\lh))\geq (2+\eps)\sqrt n|\lone)\leq \exp(-(1+o(1))K\sqrt n)\end{equation}
for all $L\in \intvl$, hence since $\lh_1$ is in this interval with probability one we can remove the conditioning and get 
\begin{equation}\P_{n,\ell}^+(\max(\lh_2,\ell(\lh))\geq (2+\eps)\sqrt n)\leq \exp(-(1+o(1))K\sqrt n).\end{equation}

Together with~\eqref{eq_lowr_bound_second_part}, this directly implies that
\begin{equation}\frac{\lh_2}{\sqrt n}\xrightarrow{p} 2\quad \text{and}\quad \frac{\ell(\lh)}{\sqrt n}\xrightarrow{p} 2\end{equation}
which is what we wanted to show.
\end{proof}

\section{The associated map model} \label{sec:mapsproofs}

In this section, $\H_{n,\ell}$ denotes the set of not necessarily connected pure Hurwitz maps on $n$ vertices and $\ell$ edges (we recall that $|\H_{n,\ell}| = H_{n,\ell}$), and $\mathbf H_{n,\ell}$ denotes a random uniform element of $\H_{n,\ell}$.

Recall that by Lemma~\ref{lem:compare_normalizations} we have, uniformly for $L\in\intvl$,
\begin{equation} \label{eq:norm_ratio_best}
\frac{Z_{n-1}[ L]}{Z_n[L]}\geq 1 - O(1/\log n).
\end{equation}

This immediately implies that, in the associated map model, most vertices are isolated.

\begin{lemma}\label{lem_isolated_vertices}
There are $n-O_p( n /\log n)$ isolated vertices in $\mathbf H_{n,\ell}$.
\end{lemma}

\begin{proof}
Choose a vertex of $\mathbf H_{n,\ell}$. The probability that this vertex is isolated is $H_{n-1,\ell}/H_{n,\ell}$, i.e. the fraction of elements of $\mathcal{H}_{n,\ell}$ remaining when a vertex is removed.   Hence, the expected number of isolated vertices in $\mathbf H_{n,\ell}$ is  $n H_{n-1,\ell}/H_{n,\ell}$. 
From~\eqref{eq_exp_suppressed}, we have (recalling the notation introduced at~\eqref{eq:znLnotation} and that $\eps = \frac{1}{100}$)  that the contribution from any partitions without $\lone$ will be exponentially suppressed, we have
\begin{equation}
    H_{n,\ell} = Z_n[L](1+ O(\exp(-n^{1-\eps}) )
\end{equation}
for some $L  \in\intvl$. Then, applying~\eqref{eq:norm_ratio_best}, this expectation is
\begin{equation}
    \frac{n H_{n-1,\ell}}{H_{n,\ell}} = \frac{nZ_{n-1}[L]}{Z_n[L]} (1+ O(\exp(-n^{1-\eps}) ) = n-O( n /\log n)
\end{equation}
where we replace the inequality in~\eqref{eq:norm_ratio_best} by the fact that the number of isolated vertices is at most $n$. From the Markov inequality, the probability that the number of isolated vertices in a uniform random element of $\mathcal{H}_{n,\ell} $ is greater than this expectation tends to 1 as $n$ tends to infinity, as required.
\end{proof}

We are now ready to prove Theorem~\ref{thm:nogiantcomp}.

\begin{proof}[Proof of Theorem~\ref{thm:nogiantcomp}]
Lemma~\ref{lem_isolated_vertices} immediately implies that all connected components of $\mathbf H_{n,\ell}$ have $O_p( n /\log n)$ vertices.

Let $\gamma <\gamma(\theta):=2^{2\theta-1}-1$. For the second part of the proof, we will show that there exists w.h.p. a component of $\mathbf H_{n,\ell}$ with at least $(\gamma(\theta)-o(1))\ell$ edges. To do so, we show that the number of maps with no connected component with more than $\gamma \ell$ edges is $o(H_{n,\ell})$, starting by estimating the number of these maps satisfying  some extra constraints.

Let $\mathfrak h\in \H_{n,\ell}$ be a map with $n/3$ isolated vertices and no connected component with more than  $n/3$ vertices, or more than $\gamma \ell$ edges. We partition the connected components of $\mathfrak h$ into two pure Hurwitz maps $\mathfrak h_1$ and $\mathfrak h_2$, where each of $\mathfrak h_1, \mathfrak h_2$ has less than $(1+\gamma)\ell/2$ edges and more than $n/3$ vertices (this is always possible since the $n/3$ isolated vertices can be distributed across the two maps any way we want).
Hence, if we let $H_{n,\ell}^{\leq \gamma}$ be the number of such maps, we have the following inequality:
\begin{equation}\label{eqY}
H_{n,\ell}^{\leq \gamma}\leq \sum_{\substack{n_1+n_2=n\\ n_1,n_2\geq n/3 }}\sum_{\substack{\ell_1+\ell_2=n\\ \ell_1,\ell_2\leq \frac{1+\gamma}2\ell}}\binom{n}{n_1}\binom{\ell}{\ell_1}H_{n_1,\ell_1}H_{n_2,\ell_2}
\end{equation}
(the binomials arise because we consider labeled objects).

We recall the asymptotic estimation of Theorem~\ref{thm:asymptotic_H_nl}, written in a convenient way for us in this proof. As long as $n,\ell\to \infty$ linearly in each other, we have
\begin{equation}\label{eqW}
\frac{H_{n,\ell}}{n!\ell!}=\exp(\ell \log \ell-n\log n-(1-\log 2)\ell+n+2\ell \log\log n+o(n))
\end{equation}
where the little-o is uniform for $\ell/n$ in any compact subset of $(0,\infty)$.

Now, take numbers $(n_1,n_2,\ell_1,\ell_2)$ as above. We can apply the asymptotic estimation \eqref{eqW} without restriction because $\frac{\ell_i}{n_i}\in \left[\frac{3(1-\gamma) }4\frac{\ell}{n},\frac{3(1+\gamma)}2\frac{\ell}{n}\right]$ and hence, for $i = 1,2$, $\ell_i,n_i\to\infty$ linearly in each other.
Therefore we obtain, since in this case $\log\log n_i=\log\log n+o(1)$ also holds,
\begin{equation}
\frac{\binom{n}{n_1}\binom{\ell}{\ell_1}H_{n_1,\ell_1}H_{n_2,\ell_2}}{H_{n,\ell}}\leq \frac{\exp\left[\ell_1\log\ell_1+\ell_2\log\ell_2-\ell\log\ell +o(n)\right]}{\exp\left[n_1\log n_1+n_2\log n_2-n\log n\right]}.
\end{equation}
Since $\ell_i\leq \frac{1+\gamma}2\ell$, we have 
\begin{equation}
(\ell_1\log\ell_1+\ell_2\log\ell_2-\ell\log\ell)\leq \ell\log\left(\frac{1+\gamma}2\right).
\end{equation}
On the other hand, we always have 
\begin{equation}
n_1\log n_1+n_2\log n_2-n\log n\geq -n\log 2,
\end{equation}
hence
\begin{equation}
\frac{\binom{n}{n_1}\binom{\ell}{\ell_1}H_{n_1,\ell_1}H_{n_2,\ell_2}}{H_{n,\ell}}\leq \exp\left[n\log 2+\ell\log\left(\frac{1+\gamma}2\right)+o(n)\right],
\end{equation}
which is exponentially small in $n$. Plugging this into~\eqref{eqY} (which has a polynomial number of summands) one obtains $H_{n,\ell}^{\leq \gamma}=o\left(H_{n,\ell}\right)$.

Let $\mathbf{E}^{<\gamma}$ denote the event that all components of $\mathbf{H}_{n,\ell}$ have less than $\gamma \ell$ edges and let $\mathbf{E}_{n/3}$ denote the even that $\mathbf{H}_{n,\ell}$ contains at least $n/3$ isolated vertices. We have
\begin{align}
    \P(\mathbf{E}^{<\gamma})=\P(\mathbf{E}^{<\gamma} \cap \mathbf{E}_{n/3})+o(1)=\frac{H_{n,\ell}^{\leq \gamma}}{H_{n,\ell}}+o(1)
    =o(1),
\end{align}
where in the first equality we use Lemma~\ref{lem_isolated_vertices}.

This shows that w.h.p., a uniform map of $\H_{n,\ell}$ has a component with more than $(\gamma(\theta)-o(1))\ell$ edges.
\end{proof}

\begin{acks}[Acknowledgments]
We thank Philippe Biane, Jérémie Bouttier and Andrea Sportiello for insightful conversations.

\end{acks}

\begin{funding}
This project has received funding from the European Research Council (ERC) under the European Union’s Horizon 2020 research and innovation programme (grant agreement No. ERC-2016-STG 716083 “CombiTop”). BL is supported by the Knut and Alice Wallenberg foundation. GC is supported by the grant ANR-19-CE48-0011 ``COMBINÉ'', GC and HW by the grant ANR-18-CE40-0033 ``DIMERS''.
\end{funding}
 
\bibliographystyle{imsart-number}
\bibliography{hurwitz}

\end{document}

%% file: for-highgenus-sim.tex
\begingroup%
  \makeatletter%
  \providecommand\color[2][]{%
    \errmessage{(Inkscape) Color is used for the text in Inkscape, but the package 'color.sty' is not loaded}%
    \renewcommand\color[2][]{}%
  }%
  \providecommand\transparent[1]{%
    \errmessage{(Inkscape) Transparency is used (non-zero) for the text in Inkscape, but the package 'transparent.sty' is not loaded}%
    \renewcommand\transparent[1]{}%
  }%
  \providecommand\rotatebox[2]{#2}%
  \newcommand*\fsize{\dimexpr\f@size pt\relax}%
  \newcommand*\lineheight[1]{\fontsize{\fsize}{#1\fsize}\selectfont}%
  \ifx\svgwidth\undefined%
    \setlength{\unitlength}{285.61994045bp}%
    \ifx\svgscale\undefined%
      \relax%
    \else%
      \setlength{\unitlength}{\unitlength * \real{\svgscale}}%
    \fi%
  \else%
    \setlength{\unitlength}{\svgwidth}%
  \fi%
  \global\let\svgwidth\undefined%
  \global\let\svgscale\undefined%
  \makeatother%
  \begin{picture}(1,0.64016584)%
    \lineheight{1}%
    \setlength\tabcolsep{0pt}%
    \put(0,0){\includegraphics[width=\unitlength,page=1]{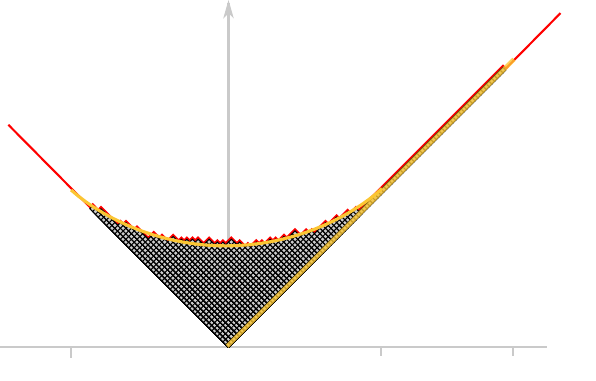}}%
    \put(0.08369219,0.00582614){\makebox(0,0)[lt]{\lineheight{1.25}\smash{\begin{tabular}[t]{l}{\scriptsize $-2\sqrt{{n}}$}\end{tabular}}}}%
    \put(0.60500994,0.0042713){\makebox(0,0)[lt]{\lineheight{1.25}\smash{\begin{tabular}[t]{l}{\scriptsize $2\sqrt{{n}}$}\end{tabular}}}}%
    \put(0.83048682,0.00629707){\makebox(0,0)[lt]{\lineheight{1.25}\smash{\begin{tabular}[t]{l}{\scriptsize $\frac{2\ell}{\log n}$}\end{tabular}}}}%
    \put(0,0){\includegraphics[width=\unitlength,page=2]{highgenus-sim-aofa-2.pdf}}%
  \end{picture}%
\endgroup%

%% file: for_hurwitz_unconnected_selection.tex
\begingroup%
  \makeatletter%
  \providecommand\color[2][]{%
    \errmessage{(Inkscape) Color is used for the text in Inkscape, but the package 'color.sty' is not loaded}%
    \renewcommand\color[2][]{}%
  }%
  \providecommand\transparent[1]{%
    \errmessage{(Inkscape) Transparency is used (non-zero) for the text in Inkscape, but the package 'transparent.sty' is not loaded}%
    \renewcommand\transparent[1]{}%
  }%
  \providecommand\rotatebox[2]{#2}%
  \newcommand*\fsize{\dimexpr\f@size pt\relax}%
  \newcommand*\lineheight[1]{\fontsize{\fsize}{#1\fsize}\selectfont}%
  \ifx\svgwidth\undefined%
    \setlength{\unitlength}{1346.3004402bp}%
    \ifx\svgscale\undefined%
      \relax%
    \else%
      \setlength{\unitlength}{\unitlength * \real{\svgscale}}%
    \fi%
  \else%
    \setlength{\unitlength}{\svgwidth}%
  \fi%
  \global\let\svgwidth\undefined%
  \global\let\svgscale\undefined%
  \makeatother%
  \begin{picture}(1,0.36961529)%
    \lineheight{1}%
    \setlength\tabcolsep{0pt}%
    \put(0,0){\includegraphics[width=\unitlength,page=1]{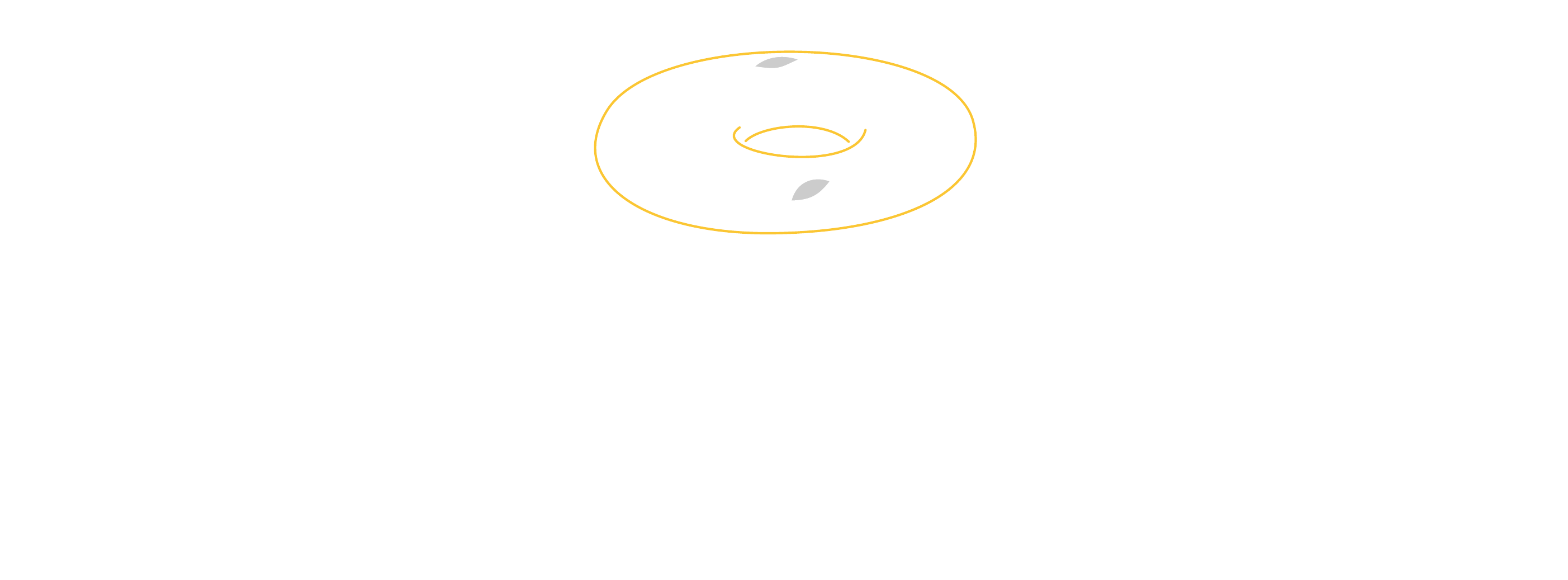}}%
    \put(0.40424279,0.30958695){\color[rgb]{0,0,0}\makebox(0,0)[lt]{\lineheight{1.25}\smash{\begin{tabular}[t]{l}{\scriptsize 1}\end{tabular}}}}%
    \put(0.46064617,0.29243096){\color[rgb]{0,0,0}\makebox(0,0)[lt]{\lineheight{1.25}\smash{\begin{tabular}[t]{l}{\scriptsize 2}\end{tabular}}}}%
    \put(0.59567195,0.28974909){\color[rgb]{0,0,0}\makebox(0,0)[lt]{\lineheight{1.25}\smash{\begin{tabular}[t]{l}{\scriptsize 3}\end{tabular}}}}%
    \put(0.54808026,0.26913965){\color[rgb]{0,0,0}\makebox(0,0)[lt]{\lineheight{1.25}\smash{\begin{tabular}[t]{l}{\scriptsize 4}\end{tabular}}}}%
    \put(0,0){\includegraphics[width=\unitlength,page=2]{hurwitz_unconnected_selection.pdf}}%
    \put(0.44282779,0.11921776){\makebox(0,0)[lt]{\lineheight{1.25}\smash{\begin{tabular}[t]{l}{\scriptsize 5}\end{tabular}}}}%
    \put(0.5525076,0.10422174){\makebox(0,0)[lt]{\lineheight{1.25}\smash{\begin{tabular}[t]{l}{\scriptsize 6}\end{tabular}}}}%
    \put(0,0){\includegraphics[width=\unitlength,page=3]{hurwitz_unconnected_selection.pdf}}%
    \put(0.48362794,0.13208628){\color[rgb]{0,0,0}\makebox(0,0)[lt]{\lineheight{1.25}\smash{\begin{tabular}[t]{l}{\scriptsize 3}\end{tabular}}}}%
    \put(0,0){\includegraphics[width=\unitlength,page=4]{hurwitz_unconnected_selection.pdf}}%
    \put(0.89722298,0.27115269){\makebox(0,0)[lt]{\lineheight{1.25}\smash{\begin{tabular}[t]{l}{\scriptsize 2}\end{tabular}}}}%
    \put(0.94391798,0.28312577){\makebox(0,0)[lt]{\lineheight{1.25}\smash{\begin{tabular}[t]{l}{\scriptsize 3}\end{tabular}}}}%
    \put(0.96321048,0.29797693){\makebox(0,0)[lt]{\lineheight{1.25}\smash{\begin{tabular}[t]{l}{\scriptsize 4}\end{tabular}}}}%
    \put(0.86993459,0.23363902){\makebox(0,0)[lt]{\lineheight{1.25}\smash{\begin{tabular}[t]{l}{\scriptsize 5}\end{tabular}}}}%
    \put(0.83622042,0.21931051){\makebox(0,0)[lt]{\lineheight{1.25}\smash{\begin{tabular}[t]{l}{\scriptsize 6}\end{tabular}}}}%
    \put(0.74434931,0.26735322){\makebox(0,0)[lt]{\lineheight{1.25}\smash{\begin{tabular}[t]{l}{\scriptsize 1}\end{tabular}}}}%
    \put(0,0){\includegraphics[width=\unitlength,page=5]{hurwitz_unconnected_selection.pdf}}%
    \put(0.85880866,0.31251664){\color[rgb]{0,0,0}\makebox(0,0)[lt]{\lineheight{1.25}\smash{\begin{tabular}[t]{l}{\scriptsize 1}\end{tabular}}}}%
    \put(0,0){\includegraphics[width=\unitlength,page=6]{hurwitz_unconnected_selection.pdf}}%
    \put(0.88191389,0.18786642){\color[rgb]{0,0,0}\makebox(0,0)[lt]{\lineheight{1.25}\smash{\begin{tabular}[t]{l}{\scriptsize 2}\end{tabular}}}}%
    \put(0,0){\includegraphics[width=\unitlength,page=7]{hurwitz_unconnected_selection.pdf}}%
    \put(0.18518172,0.22134368){\color[rgb]{0,0,0}\makebox(0,0)[lt]{\lineheight{1.25}\smash{\begin{tabular}[t]{l}{\scriptsize 3}\end{tabular}}}}%
    \put(0,0){\includegraphics[width=\unitlength,page=8]{hurwitz_unconnected_selection.pdf}}%
    \put(0.05020197,0.26141962){\color[rgb]{0,0,0}\makebox(0,0)[lt]{\lineheight{1.25}\smash{\begin{tabular}[t]{l}{\scriptsize 1}\end{tabular}}}}%
    \put(0,0){\includegraphics[width=\unitlength,page=9]{hurwitz_unconnected_selection.pdf}}%
    \put(0.09885457,0.17604257){\color[rgb]{0,0,0}\makebox(0,0)[lt]{\lineheight{1.25}\smash{\begin{tabular}[t]{l}{\scriptsize 2}\end{tabular}}}}%
    \put(0,0){\includegraphics[width=\unitlength,page=10]{hurwitz_unconnected_selection.pdf}}%
    \put(0.23942929,0.1357809){\color[rgb]{0,0,0}\makebox(0,0)[lt]{\lineheight{1.25}\smash{\begin{tabular}[t]{l}{\scriptsize 4}\end{tabular}}}}%
    \put(0,0){\includegraphics[width=\unitlength,page=11]{hurwitz_unconnected_selection.pdf}}%
    \put(0.02881864,0.21335965){\makebox(0,0)[lt]{\lineheight{1.25}\smash{\begin{tabular}[t]{l}{\scriptsize 1}\end{tabular}}}}%
    \put(0.09059916,0.2468241){\makebox(0,0)[lt]{\lineheight{1.25}\smash{\begin{tabular}[t]{l}{\scriptsize 6}\end{tabular}}}}%
    \put(0.15958742,0.17165781){\makebox(0,0)[lt]{\lineheight{1.25}\smash{\begin{tabular}[t]{l}{\scriptsize 2}\end{tabular}}}}%
    \put(0.20180409,0.16290557){\makebox(0,0)[lt]{\lineheight{1.25}\smash{\begin{tabular}[t]{l}{\scriptsize 3}\end{tabular}}}}%
    \put(0.25174337,0.20306291){\makebox(0,0)[lt]{\lineheight{1.25}\smash{\begin{tabular}[t]{l}{\scriptsize 4}\end{tabular}}}}%
    \put(0.13950873,0.23240865){\makebox(0,0)[lt]{\lineheight{1.25}\smash{\begin{tabular}[t]{l}{\scriptsize 5}\end{tabular}}}}%
    \put(0,0){\includegraphics[width=\unitlength,page=12]{hurwitz_unconnected_selection.pdf}}%
    \put(0.49411427,0.31642323){\color[rgb]{0,0,0}\makebox(0,0)[lt]{\lineheight{1.25}\smash{\begin{tabular}[t]{l}{\scriptsize 1}\end{tabular}}}}%
    \put(0,0){\includegraphics[width=\unitlength,page=13]{hurwitz_unconnected_selection.pdf}}%
    \put(0.51897759,0.23590305){\color[rgb]{0,0,0}\makebox(0,0)[lt]{\lineheight{1.25}\smash{\begin{tabular}[t]{l}{\scriptsize 2}\end{tabular}}}}%
    \put(0,0){\includegraphics[width=\unitlength,page=14]{hurwitz_unconnected_selection.pdf}}%
    \put(0.7924123,0.10781465){\color[rgb]{0,0,0}\makebox(0,0)[lt]{\lineheight{1.25}\smash{\begin{tabular}[t]{l}{\scriptsize 3}\end{tabular}}}}%
    \put(0,0){\includegraphics[width=\unitlength,page=15]{hurwitz_unconnected_selection.pdf}}%
    \put(0.50893961,0.05294202){\color[rgb]{0,0,0}\makebox(0,0)[lt]{\lineheight{1.25}\smash{\begin{tabular}[t]{l}{\scriptsize 4}\end{tabular}}}}%
    \put(0,0){\includegraphics[width=\unitlength,page=16]{hurwitz_unconnected_selection.pdf}}%
    \put(0.92655562,0.07492067){\color[rgb]{0,0,0}\makebox(0,0)[lt]{\lineheight{1.25}\smash{\begin{tabular}[t]{l}{\scriptsize 4}\end{tabular}}}}%
  \end{picture}%
\endgroup%

%% file: for-yd-french.tex
\begingroup%
  \makeatletter%
  \providecommand\color[2][]{%
    \errmessage{(Inkscape) Color is used for the text in Inkscape, but the package 'color.sty' is not loaded}%
    \renewcommand\color[2][]{}%
  }%
  \providecommand\transparent[1]{%
    \errmessage{(Inkscape) Transparency is used (non-zero) for the text in Inkscape, but the package 'transparent.sty' is not loaded}%
    \renewcommand\transparent[1]{}%
  }%
  \providecommand\rotatebox[2]{#2}%
  \newcommand*\fsize{\dimexpr\f@size pt\relax}%
  \newcommand*\lineheight[1]{\fontsize{\fsize}{#1\fsize}\selectfont}%
  \ifx\svgwidth\undefined%
    \setlength{\unitlength}{172.86676191bp}%
    \ifx\svgscale\undefined%
      \relax%
    \else%
      \setlength{\unitlength}{\unitlength * \real{\svgscale}}%
    \fi%
  \else%
    \setlength{\unitlength}{\svgwidth}%
  \fi%
  \global\let\svgwidth\undefined%
  \global\let\svgscale\undefined%
  \makeatother%
  \begin{picture}(1,0.78306973)%
    \lineheight{1}%
    \setlength\tabcolsep{0pt}%
    \put(0,0){\includegraphics[width=\unitlength,page=1]{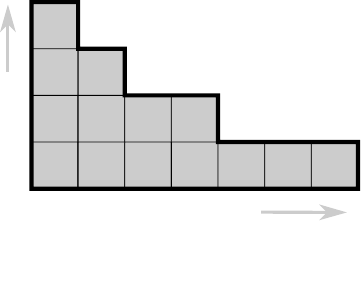}}%
    \put(0.00583738,0.48042404){\makebox(0,0)[lt]{\lineheight{1.25}\smash{\begin{tabular}[t]{l}$i$\end{tabular}}}}%
    \put(0.63514952,0.1715053){\makebox(0,0)[lt]{\lineheight{1.25}\smash{\begin{tabular}[t]{l}$j$\end{tabular}}}}%
    \put(0,0){\includegraphics[width=\unitlength,page=2]{yd-french.pdf}}%
  \end{picture}%
\endgroup%

%% file: for-yd-russian.tex
\begingroup%
  \makeatletter%
  \providecommand\color[2][]{%
    \errmessage{(Inkscape) Color is used for the text in Inkscape, but the package 'color.sty' is not loaded}%
    \renewcommand\color[2][]{}%
  }%
  \providecommand\transparent[1]{%
    \errmessage{(Inkscape) Transparency is used (non-zero) for the text in Inkscape, but the package 'transparent.sty' is not loaded}%
    \renewcommand\transparent[1]{}%
  }%
  \providecommand\rotatebox[2]{#2}%
  \newcommand*\fsize{\dimexpr\f@size pt\relax}%
  \newcommand*\lineheight[1]{\fontsize{\fsize}{#1\fsize}\selectfont}%
  \ifx\svgwidth\undefined%
    \setlength{\unitlength}{276.20183085bp}%
    \ifx\svgscale\undefined%
      \relax%
    \else%
      \setlength{\unitlength}{\unitlength * \real{\svgscale}}%
    \fi%
  \else%
    \setlength{\unitlength}{\svgwidth}%
  \fi%
  \global\let\svgwidth\undefined%
  \global\let\svgscale\undefined%
  \makeatother%
  \begin{picture}(1,0.68192857)%
    \lineheight{1}%
    \setlength\tabcolsep{0pt}%
    \put(0,0){\includegraphics[width=\unitlength,page=1]{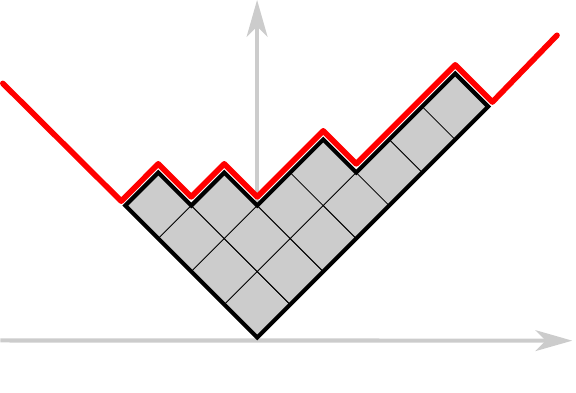}}%
    \put(0.87098288,0.01064241){\makebox(0,0)[lt]{\lineheight{1.25}\smash{\begin{tabular}[t]{l}$x$\end{tabular}}}}%
    \put(0.45814002,0.57404397){\makebox(0,0)[lt]{\lineheight{1.25}\smash{\begin{tabular}[t]{l}$\psi_\lambda(x)$\end{tabular}}}}%
  \end{picture}%
\endgroup%

%% file: for-yd-conjugate.tex
\begingroup%
  \makeatletter%
  \providecommand\color[2][]{%
    \errmessage{(Inkscape) Color is used for the text in Inkscape, but the package 'color.sty' is not loaded}%
    \renewcommand\color[2][]{}%
  }%
  \providecommand\transparent[1]{%
    \errmessage{(Inkscape) Transparency is used (non-zero) for the text in Inkscape, but the package 'transparent.sty' is not loaded}%
    \renewcommand\transparent[1]{}%
  }%
  \providecommand\rotatebox[2]{#2}%
  \newcommand*\fsize{\dimexpr\f@size pt\relax}%
  \newcommand*\lineheight[1]{\fontsize{\fsize}{#1\fsize}\selectfont}%
  \ifx\svgwidth\undefined%
    \setlength{\unitlength}{299.29145213bp}%
    \ifx\svgscale\undefined%
      \relax%
    \else%
      \setlength{\unitlength}{\unitlength * \real{\svgscale}}%
    \fi%
  \else%
    \setlength{\unitlength}{\svgwidth}%
  \fi%
  \global\let\svgwidth\undefined%
  \global\let\svgscale\undefined%
  \makeatother%
  \begin{picture}(1,0.62931941)%
    \lineheight{1}%
    \setlength\tabcolsep{0pt}%
    \put(0,0){\includegraphics[width=\unitlength,page=1]{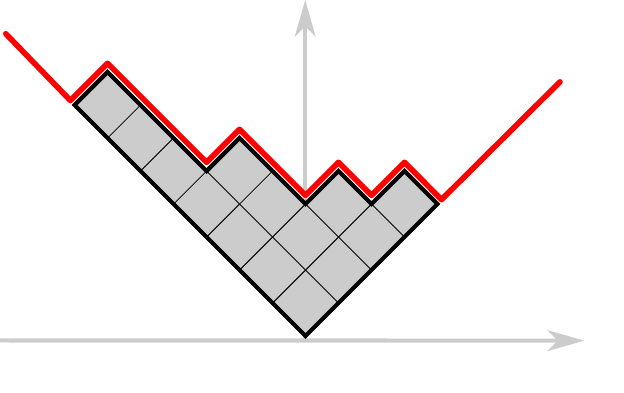}}%
    \put(0.88093624,0.00982137){\makebox(0,0)[lt]{\lineheight{1.25}\smash{\begin{tabular}[t]{l}$x$\end{tabular}}}}%
    \put(0.50182271,0.52975789){\makebox(0,0)[lt]{\lineheight{1.25}\smash{\begin{tabular}[t]{l}$\psi_{\lambda^\prime}(x)$\end{tabular}}}}%
  \end{picture}%
\endgroup%

%% file: for-frenchyd-syt.tex
\begingroup%
  \makeatletter%
  \providecommand\color[2][]{%
    \errmessage{(Inkscape) Color is used for the text in Inkscape, but the package 'color.sty' is not loaded}%
    \renewcommand\color[2][]{}%
  }%
  \providecommand\transparent[1]{%
    \errmessage{(Inkscape) Transparency is used (non-zero) for the text in Inkscape, but the package 'transparent.sty' is not loaded}%
    \renewcommand\transparent[1]{}%
  }%
  \providecommand\rotatebox[2]{#2}%
  \newcommand*\fsize{\dimexpr\f@size pt\relax}%
  \newcommand*\lineheight[1]{\fontsize{\fsize}{#1\fsize}\selectfont}%
  \ifx\svgwidth\undefined%
    \setlength{\unitlength}{225.41421356bp}%
    \ifx\svgscale\undefined%
      \relax%
    \else%
      \setlength{\unitlength}{\unitlength * \real{\svgscale}}%
    \fi%
  \else%
    \setlength{\unitlength}{\svgwidth}%
  \fi%
  \global\let\svgwidth\undefined%
  \global\let\svgscale\undefined%
  \makeatother%
  \begin{picture}(1,0.71607824)%
    \lineheight{1}%
    \setlength\tabcolsep{0pt}%
    \put(0.28640672,0.28463754){\color[rgb]{0,0,0}\makebox(0,0)[lt]{\lineheight{1.25}\smash{\begin{tabular}[t]{l}7\end{tabular}}}}%
    \put(0,0){\includegraphics[width=\unitlength,page=1]{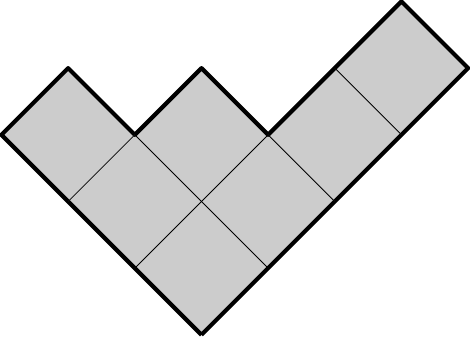}}%
    \put(0.49455161,0.16861355){\color[rgb]{0,0,0}\rotatebox{45}{\makebox(0,0)[lt]{\lineheight{1.25}\smash{\begin{tabular}[t]{l}$<$\end{tabular}}}}}%
    \put(0.63579042,0.30985237){\color[rgb]{0,0,0}\rotatebox{45}{\makebox(0,0)[lt]{\lineheight{1.25}\smash{\begin{tabular}[t]{l}$<$\end{tabular}}}}}%
    \put(0.77877509,0.45283703){\color[rgb]{0,0,0}\rotatebox{45}{\makebox(0,0)[lt]{\lineheight{1.25}\smash{\begin{tabular}[t]{l}$<$\end{tabular}}}}}%
    \put(0.34630415,0.31686102){\color[rgb]{0,0,0}\rotatebox{45}{\makebox(0,0)[lt]{\lineheight{1.25}\smash{\begin{tabular}[t]{l}$<$\end{tabular}}}}}%
    \put(0.25389258,0.33060054){\color[rgb]{0,0,0}\rotatebox{135}{\makebox(0,0)[lt]{\lineheight{1.25}\smash{\begin{tabular}[t]{l}$<$\end{tabular}}}}}%
    \put(0.39584298,0.18865014){\color[rgb]{0,0,0}\rotatebox{135}{\makebox(0,0)[lt]{\lineheight{1.25}\smash{\begin{tabular}[t]{l}$<$\end{tabular}}}}}%
    \put(0.53812912,0.33093627){\color[rgb]{0,0,0}\rotatebox{135}{\makebox(0,0)[lt]{\lineheight{1.25}\smash{\begin{tabular}[t]{l}$<$\end{tabular}}}}}%
    \put(0.11786209,0.41605692){\makebox(0,0)[lt]{\lineheight{1.25}\smash{\begin{tabular}[t]{l}4\end{tabular}}}}%
    \put(0.26542024,0.26300798){\makebox(0,0)[lt]{\lineheight{1.25}\smash{\begin{tabular}[t]{l}3\end{tabular}}}}%
    \put(0.41257758,0.11983907){\makebox(0,0)[lt]{\lineheight{1.25}\smash{\begin{tabular}[t]{l}1\end{tabular}}}}%
    \put(0.56119714,0.26269389){\makebox(0,0)[lt]{\lineheight{1.25}\smash{\begin{tabular}[t]{l}2\end{tabular}}}}%
    \put(0.70846299,0.41637101){\makebox(0,0)[lt]{\lineheight{1.25}\smash{\begin{tabular}[t]{l}5\end{tabular}}}}%
    \put(0.41281585,0.41607858){\makebox(0,0)[lt]{\lineheight{1.25}\smash{\begin{tabular}[t]{l}6\end{tabular}}}}%
    \put(0.85594512,0.56355475){\makebox(0,0)[lt]{\lineheight{1.25}\smash{\begin{tabular}[t]{l}7\end{tabular}}}}%
  \end{picture}%
\endgroup%

%% file: for-frenchyd-hooks.tex
\begingroup%
  \makeatletter%
  \providecommand\color[2][]{%
    \errmessage{(Inkscape) Color is used for the text in Inkscape, but the package 'color.sty' is not loaded}%
    \renewcommand\color[2][]{}%
  }%
  \providecommand\transparent[1]{%
    \errmessage{(Inkscape) Transparency is used (non-zero) for the text in Inkscape, but the package 'transparent.sty' is not loaded}%
    \renewcommand\transparent[1]{}%
  }%
  \providecommand\rotatebox[2]{#2}%
  \newcommand*\fsize{\dimexpr\f@size pt\relax}%
  \newcommand*\lineheight[1]{\fontsize{\fsize}{#1\fsize}\selectfont}%
  \ifx\svgwidth\undefined%
    \setlength{\unitlength}{225.41421356bp}%
    \ifx\svgscale\undefined%
      \relax%
    \else%
      \setlength{\unitlength}{\unitlength * \real{\svgscale}}%
    \fi%
  \else%
    \setlength{\unitlength}{\svgwidth}%
  \fi%
  \global\let\svgwidth\undefined%
  \global\let\svgscale\undefined%
  \makeatother%
  \begin{picture}(1,0.71607824)%
    \lineheight{1}%
    \setlength\tabcolsep{0pt}%
    \put(0.28640672,0.28463754){\color[rgb]{0,0,0}\makebox(0,0)[lt]{\lineheight{1.25}\smash{\begin{tabular}[t]{l}7\end{tabular}}}}%
    \put(0,0){\includegraphics[width=\unitlength,page=1]{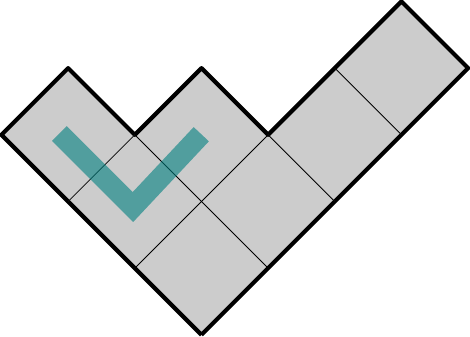}}%
    \put(0.11786209,0.41605692){\makebox(0,0)[lt]{\lineheight{1.25}\smash{\begin{tabular}[t]{l}1\end{tabular}}}}%
    \put(0.26542024,0.26300798){\makebox(0,0)[lt]{\lineheight{1.25}\smash{\begin{tabular}[t]{l}3\end{tabular}}}}%
    \put(0.41257758,0.11983907){\makebox(0,0)[lt]{\lineheight{1.25}\smash{\begin{tabular}[t]{l}6\end{tabular}}}}%
    \put(0.56119714,0.26269389){\makebox(0,0)[lt]{\lineheight{1.25}\smash{\begin{tabular}[t]{l}4\end{tabular}}}}%
    \put(0.70846299,0.41637101){\makebox(0,0)[lt]{\lineheight{1.25}\smash{\begin{tabular}[t]{l}2\end{tabular}}}}%
    \put(0.41281585,0.41607858){\makebox(0,0)[lt]{\lineheight{1.25}\smash{\begin{tabular}[t]{l}1\end{tabular}}}}%
    \put(0.85594512,0.56355475){\makebox(0,0)[lt]{\lineheight{1.25}\smash{\begin{tabular}[t]{l}1\end{tabular}}}}%
  \end{picture}%
\endgroup%

%% file: for-yd-csum.tex
\begingroup%
  \makeatletter%
  \providecommand\color[2][]{%
    \errmessage{(Inkscape) Color is used for the text in Inkscape, but the package 'color.sty' is not loaded}%
    \renewcommand\color[2][]{}%
  }%
  \providecommand\transparent[1]{%
    \errmessage{(Inkscape) Transparency is used (non-zero) for the text in Inkscape, but the package 'transparent.sty' is not loaded}%
    \renewcommand\transparent[1]{}%
  }%
  \providecommand\rotatebox[2]{#2}%
  \newcommand*\fsize{\dimexpr\f@size pt\relax}%
  \newcommand*\lineheight[1]{\fontsize{\fsize}{#1\fsize}\selectfont}%
  \ifx\svgwidth\undefined%
    \setlength{\unitlength}{225.41421356bp}%
    \ifx\svgscale\undefined%
      \relax%
    \else%
      \setlength{\unitlength}{\unitlength * \real{\svgscale}}%
    \fi%
  \else%
    \setlength{\unitlength}{\svgwidth}%
  \fi%
  \global\let\svgwidth\undefined%
  \global\let\svgscale\undefined%
  \makeatother%
  \begin{picture}(1,0.71607824)%
    \lineheight{1}%
    \setlength\tabcolsep{0pt}%
    \put(0.28640672,0.28463754){\color[rgb]{0,0,0}\makebox(0,0)[lt]{\lineheight{1.25}\smash{\begin{tabular}[t]{l}7\end{tabular}}}}%
    \put(0,0){\includegraphics[width=\unitlength,page=1]{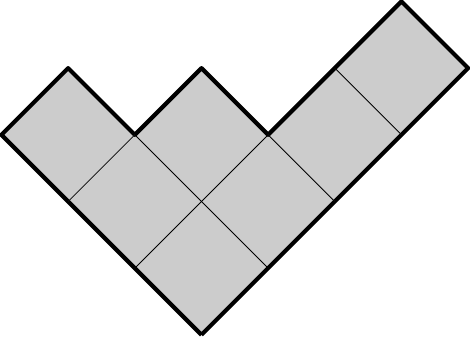}}%
    \put(0.11786209,0.41605692){\makebox(0,0)[lt]{\lineheight{1.25}\smash{\begin{tabular}[t]{l}-2\end{tabular}}}}%
    \put(0.26542024,0.26300798){\makebox(0,0)[lt]{\lineheight{1.25}\smash{\begin{tabular}[t]{l}-1\end{tabular}}}}%
    \put(0.41257758,0.11983907){\makebox(0,0)[lt]{\lineheight{1.25}\smash{\begin{tabular}[t]{l}0\end{tabular}}}}%
    \put(0.56119714,0.26269389){\makebox(0,0)[lt]{\lineheight{1.25}\smash{\begin{tabular}[t]{l}1\end{tabular}}}}%
    \put(0.70846299,0.41637101){\makebox(0,0)[lt]{\lineheight{1.25}\smash{\begin{tabular}[t]{l}2\end{tabular}}}}%
    \put(0.41281585,0.41607858){\makebox(0,0)[lt]{\lineheight{1.25}\smash{\begin{tabular}[t]{l}0\end{tabular}}}}%
    \put(0.85594512,0.56355475){\makebox(0,0)[lt]{\lineheight{1.25}\smash{\begin{tabular}[t]{l}3\end{tabular}}}}%
  \end{picture}%
\endgroup%

%% file: for_diagram_comparison.tex
\begingroup%
  \makeatletter%
  \providecommand\color[2][]{%
    \errmessage{(Inkscape) Color is used for the text in Inkscape, but the package 'color.sty' is not loaded}%
    \renewcommand\color[2][]{}%
  }%
  \providecommand\transparent[1]{%
    \errmessage{(Inkscape) Transparency is used (non-zero) for the text in Inkscape, but the package 'transparent.sty' is not loaded}%
    \renewcommand\transparent[1]{}%
  }%
  \providecommand\rotatebox[2]{#2}%
  \newcommand*\fsize{\dimexpr\f@size pt\relax}%
  \newcommand*\lineheight[1]{\fontsize{\fsize}{#1\fsize}\selectfont}%
  \ifx\svgwidth\undefined%
    \setlength{\unitlength}{618.02655404bp}%
    \ifx\svgscale\undefined%
      \relax%
    \else%
      \setlength{\unitlength}{\unitlength * \real{\svgscale}}%
    \fi%
  \else%
    \setlength{\unitlength}{\svgwidth}%
  \fi%
  \global\let\svgwidth\undefined%
  \global\let\svgscale\undefined%
  \makeatother%
  \begin{picture}(1,0.18133757)%
    \lineheight{1}%
    \setlength\tabcolsep{0pt}%
    \put(0,0){\includegraphics[width=\unitlength,page=1]{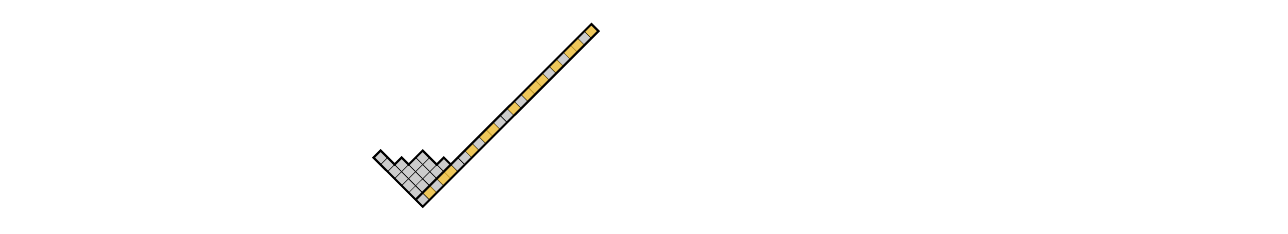}}%
    \put(0.42511959,0.09158322){\makebox(0,0)[lt]{\lineheight{1.25}\smash{\begin{tabular}[t]{l}$M$\end{tabular}}}}%
    \put(0,0){\includegraphics[width=\unitlength,page=2]{diagram_comparison.pdf}}%
    \put(0.37566546,0.04451586){\makebox(0,0)[lt]{\lineheight{1.25}\smash{\begin{tabular}[t]{l}$n^{1-\varepsilon}$\end{tabular}}}}%
    \put(0,0){\includegraphics[width=\unitlength,page=3]{diagram_comparison.pdf}}%
    \put(0.63643377,0.06864144){\makebox(0,0)[lt]{\lineheight{1.25}\smash{\begin{tabular}[t]{l}$M-m$\end{tabular}}}}%
    \put(0,0){\includegraphics[width=\unitlength,page=4]{diagram_comparison.pdf}}%
    \put(0.61758307,0.04824234){\makebox(0,0)[lt]{\lineheight{1.25}\smash{\begin{tabular}[t]{l}$n^{1-\varepsilon}$\end{tabular}}}}%
    \put(0,0){\includegraphics[width=\unitlength,page=5]{diagram_comparison.pdf}}%
    \put(0.53438718,0.01355989){\makebox(0,0)[lt]{\lineheight{1.25}\smash{\begin{tabular}[t]{l}$n^{\varepsilon}$\end{tabular}}}}%
    \put(0,0){\includegraphics[width=\unitlength,page=6]{diagram_comparison.pdf}}%
    \put(0.27215064,0.07388829){\makebox(0,0)[lt]{\lineheight{1.25}\smash{\begin{tabular}[t]{l}$\lambda^- = \mu$\end{tabular}}}}%
    \put(0.49951272,0.09089504){\makebox(0,0)[lt]{\lineheight{1.25}\smash{\begin{tabular}[t]{l}$\lambda^- = \mu$\end{tabular}}}}%
    \put(0,0){\includegraphics[width=\unitlength,page=7]{diagram_comparison.pdf}}%
    \put(0.91636176,0.06604009){\makebox(0,0)[lt]{\lineheight{1.25}\smash{\begin{tabular}[t]{l}$n^{1-\varepsilon}$\end{tabular}}}}%
    \put(0.74366714,0.05893916){\makebox(0,0)[lt]{\lineheight{1.25}\smash{\begin{tabular}[t]{l}$n^{1-\varepsilon}$\end{tabular}}}}%
    \put(0.7156956,0.12038572){\makebox(0,0)[lt]{\lineheight{1.25}\smash{\begin{tabular}[t]{l}$|\nu^+|=m$\end{tabular}}}}%
    \put(0.949377,0.09621555){\makebox(0,0)[lt]{\lineheight{1.25}\smash{\begin{tabular}[t]{l}$L$\end{tabular}}}}%
    \put(0,0){\includegraphics[width=\unitlength,page=8]{diagram_comparison.pdf}}%
    \put(0.0166641,0.10434381){\makebox(0,0)[lt]{\lineheight{1.25}\smash{\begin{tabular}[t]{l}$\lambda^- = \mu$\end{tabular}}}}%
    \put(0.13520995,0.06376362){\makebox(0,0)[lt]{\lineheight{1.25}\smash{\begin{tabular}[t]{l}$M-m$\end{tabular}}}}%
    \put(0,0){\includegraphics[width=\unitlength,page=9]{diagram_comparison.pdf}}%
    \put(0.11717262,0.16090956){\makebox(0,0)[lt]{\lineheight{1.25}\smash{\begin{tabular}[t]{l}$|\lambda^+| = M$\end{tabular}}}}%
    \put(0,0){\includegraphics[width=\unitlength,page=10]{diagram_comparison.pdf}}%
    \put(0.17179706,0.09659875){\makebox(0,0)[lt]{\lineheight{1.25}\smash{\begin{tabular}[t]{l}$n^{1-\varepsilon}$\end{tabular}}}}%
    \put(0,0){\includegraphics[width=\unitlength,page=11]{diagram_comparison.pdf}}%
    \put(0.83684121,0.100804){\makebox(0,0)[lt]{\lineheight{1.25}\smash{\begin{tabular}[t]{l}$\nu^-$\end{tabular}}}}%
  \end{picture}%
\endgroup%

%% file: for-remove-2nd-part.tex
\begingroup%
  \makeatletter%
  \providecommand\color[2][]{%
    \errmessage{(Inkscape) Color is used for the text in Inkscape, but the package 'color.sty' is not loaded}%
    \renewcommand\color[2][]{}%
  }%
  \providecommand\transparent[1]{%
    \errmessage{(Inkscape) Transparency is used (non-zero) for the text in Inkscape, but the package 'transparent.sty' is not loaded}%
    \renewcommand\transparent[1]{}%
  }%
  \providecommand\rotatebox[2]{#2}%
  \newcommand*\fsize{\dimexpr\f@size pt\relax}%
  \newcommand*\lineheight[1]{\fontsize{\fsize}{#1\fsize}\selectfont}%
  \ifx\svgwidth\undefined%
    \setlength{\unitlength}{695.29529403bp}%
    \ifx\svgscale\undefined%
      \relax%
    \else%
      \setlength{\unitlength}{\unitlength * \real{\svgscale}}%
    \fi%
  \else%
    \setlength{\unitlength}{\svgwidth}%
  \fi%
  \global\let\svgwidth\undefined%
  \global\let\svgscale\undefined%
  \makeatother%
  \begin{picture}(1,0.16735201)%
    \lineheight{1}%
    \setlength\tabcolsep{0pt}%
    \put(-0.00061435,0.0862073){\makebox(0,0)[lt]{\lineheight{1.25}\smash{\begin{tabular}[t]{l}{\scriptsize $\mu \vdash n-L-k$}\end{tabular}}}}%
    \put(0.24678358,0.0862073){\makebox(0,0)[lt]{\lineheight{1.25}\smash{\begin{tabular}[t]{l}{\scriptsize $\mu \vdash n-L-k$}\end{tabular}}}}%
    \put(0,0){\includegraphics[width=\unitlength,page=1]{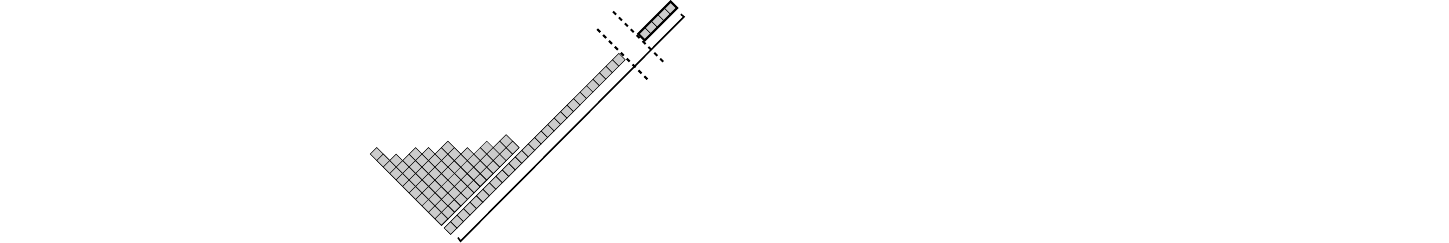}}%
    \put(0.39026419,0.05602125){\makebox(0,0)[lt]{\lineheight{1.25}\smash{\begin{tabular}[t]{l}{\scriptsize $L$}\end{tabular}}}}%
    \put(0,0){\includegraphics[width=\unitlength,page=2]{remove-2nd-part.pdf}}%
    \put(0.14273777,0.05602125){\makebox(0,0)[lt]{\lineheight{1.25}\smash{\begin{tabular}[t]{l}{\scriptsize $L$}\end{tabular}}}}%
    \put(0.12472585,0.10359118){\makebox(0,0)[lt]{\lineheight{1.25}\smash{\begin{tabular}[t]{l}{\scriptsize $k$}\end{tabular}}}}%
    \put(0,0){\includegraphics[width=\unitlength,page=3]{remove-2nd-part.pdf}}%
    \put(0.50993521,0.07712241){\makebox(0,0)[lt]{\lineheight{1.25}\smash{\begin{tabular}[t]{l}{\scriptsize $k$}\end{tabular}}}}%
    \put(0,0){\includegraphics[width=\unitlength,page=4]{remove-2nd-part.pdf}}%
    \put(0.67572863,0.05602125){\makebox(0,0)[lt]{\lineheight{1.25}\smash{\begin{tabular}[t]{l}{\scriptsize $L$}\end{tabular}}}}%
    \put(0.91943313,0.05602125){\makebox(0,0)[lt]{\lineheight{1.25}\smash{\begin{tabular}[t]{l}{\scriptsize $L$}\end{tabular}}}}%
    \put(0,0){\includegraphics[width=\unitlength,page=5]{remove-2nd-part.pdf}}%
    \put(0.54181345,0.08536001){\makebox(0,0)[lt]{\lineheight{1.25}\smash{\begin{tabular}[t]{l}{\scriptsize $\mu^\prime \vdash n-L-k$}\end{tabular}}}}%
    \put(0.78921138,0.08536001){\makebox(0,0)[lt]{\lineheight{1.25}\smash{\begin{tabular}[t]{l}{\scriptsize $\mu^\prime \vdash n-L-k$}\end{tabular}}}}%
  \end{picture}%
\endgroup%